\documentclass[10pt,leqno]{amsart}

\usepackage[T1]{fontenc}
\usepackage[utf8]{inputenc}
\usepackage{lmodern}
\usepackage{amsmath,amssymb,amsthm,mathtools}
\usepackage{mathrsfs}
\usepackage{enumitem}
\usepackage[a4paper,hmargin=1.5in,vmargin=1in]{geometry}
\usepackage{graphicx}
\usepackage{tikz}
\usepackage{pgfplots}
\pgfplotsset{compat=1.18}
\usepgfplotslibrary{groupplots}
\usepackage{microtype}
\usepackage{cite}
\usepackage{xcolor}
\usepackage[colorlinks=true,linkcolor=blue,citecolor=blue,
urlcolor=blue,bookmarksnumbered=true]{hyperref}

\numberwithin{equation}{section}
\setlist[enumerate]{leftmargin=*,label=(\roman*),itemsep=2pt,topsep=4pt}
\allowdisplaybreaks[2]

\newcommand{\R}{\mathbb R}
\newcommand{\Sph}{\mathbb S}
\newcommand{\calC}{\mathcal C}
\newcommand{\Ga}{\Gamma}
\newcommand{\dd}{\,d}
\newcommand{\ip}[2]{\langle #1,#2\rangle}
\newcommand{\binomc}{\binom{n-1}{r-1}}

\newtheorem{maintheorem}{Theorem}

\theoremstyle{plain}
\newtheorem{theorem}{Theorem}[section]
\newtheorem{proposition}[theorem]{Proposition}
\newtheorem{lemma}[theorem]{Lemma}
\newtheorem{corollary}[theorem]{Corollary}
\theoremstyle{definition}

\newtheorem{question}[theorem]{Open question}
\theoremstyle{definition}
\newtheorem{remark}[theorem]{Remark}

\title[Conical ends of $S_r$-shrinkers]
{Asymptotically conical shrinking ends of the $r$-mean curvature flow: rotational construction and uniqueness}

\author{Hil\'ario Alencar}
\address{Instituto de Matem\'atica, Universidade Federal de Alagoas, Macei\'o, 57072-900, Brazil}
\email{hilario@mat.ufal.br}
\author{M\'arcio Batista}
\address{CPMAT -- Instituto de Matem\'atica, Universidade Federal de Alagoas, Macei\'o, 57072-900, Brazil}
\email{mhbs@mat.ufal.br}
\author{Greg\'orio Silva Neto}
\address{Instituto de Matem\'atica, Universidade Federal de Alagoas, Macei\'o, 57072-900, Brazil}
\email{gregorio@im.ufal.br}



\begin{document}

\begin{abstract}
For $1\le r\le n-1$ we study shrinking self-similar ends of the
$r$-mean curvature flow, $S_r=c\langle X,N\rangle$ with $c<0$, that are
asymptotic to a cone. For every rotational cone of positive slope we
construct a smooth embedded rotational end and prove that its profile is
the only positive solution of the profile equation with the prescribed
slope at infinity. Its angular defect $xu'-u$ is strictly negative, which
yields $r$-admissibility, and the profile has a complete asymptotic
expansion for $u(x)/x$ in powers of $x^{-(r+1)}$, valid after arbitrary
differentiation. The profiles depend real analytically on the slope and, on sufficiently
distant exterior regions, foliate the image of the corresponding family. Our main result removes the symmetry
assumption: two ends that converge smoothly, with multiplicity one, to the
same circular cone coincide outside a compact set. The proof transforms
the equation for their difference into an almost-eigenfunction inequality
for a drift Laplacian on an auxiliary conical metric and applies
Bernstein's asymptotic analysis on weakly conical ends. The argument
extends to strictly $r$-admissible cones whose links have constant
principal curvatures.
\end{abstract}

\maketitle

\section{Introduction}

Let $X_0:\Sigma^n\to\R^{n+1}$ be an immersed hypersurface with principal
curvatures $\lambda_1,\ldots,\lambda_n$. For $1\le r\le n$, let
\[
S_r(\lambda)=\sum_{i_1<\cdots<i_r}\lambda_{i_1}\cdots\lambda_{i_r}
\]
be the $r$th elementary symmetric function. The $r$-mean curvature flow is
the curvature evolution
\begin{equation}\label{eq:flow}
\left(\frac{\partial X}{\partial t}\right)^\perp=S_rN.
\end{equation}
For $r=1$ this is the mean curvature flow. For $r>1$ the flow is fully
nonlinear and is parabolic on the G\aa rding cone
\[
\Ga_r=\{\lambda\in\R^n:S_1(\lambda)>0,\ldots,S_r(\lambda)>0\}.
\]
The study of contracting convex hypersurfaces by homogeneous curvature
functions goes back to Andrews
\cite{Andrews1994,Andrews2000,Andrews2007}; see also
\cite{AndrewsGuanNi2016,AndrewsMcCoy2012,BrendleChoiDaskalopoulos2017,McCoy2011,McCoy2021}
for powers of curvature functions, pinching, convergence, and self-similar
solutions. The restriction to an ellipticity cone is intrinsic to this
theory: outside it the linearization of the speed need not be positive.

If $X(\cdot,t)=\phi(t)X_0$ is a homothetic solution of
\eqref{eq:flow}, then $X_0$ satisfies
\begin{equation}\label{eq:shrinker}
S_r=c\ip{X}{N}
\end{equation}
for a constant $c$. Conversely, an $r$-admissible solution of
\eqref{eq:shrinker} generates the homothetic solution with
\[
\phi(t)=\bigl(1+(r+1)ct\bigr)^{1/(r+1)}.
\]
Thus $c<0$ corresponds to a shrinker with extinction time
$-1/((r+1)c)$. Under $X\mapsto\rho X$ the coefficient changes according
to $c\mapsto\rho^{-(r+1)}c$.

For $r=1$, self-shrinkers model Type-I singularities of the mean curvature
flow; see Huisken \cite{Huisken1984,Huisken1990} and Colding--Minicozzi
\cite{ColdingMinicozzi2012}. Rotational examples have played a prominent
role since Angenent's shrinking torus \cite{Angenent1992}. Kleene and
M\o ller \cite[Theorem~1, p.~3943]{KleeneMoller2014} constructed a one-parameter family of
properly embedded rotational self-shrinking ends asymptotic to rotational
cones, usually called trumpets, and classified the noncompact rotational
self-shrinking ends \cite[Theorem~3, p.~3947]{KleeneMoller2014}. Wang \cite[Theorem~1.3, p.~614]{Wang2014} proved that a smooth
properly embedded self-shrinking end asymptotic to a regular cone is
unique, without any symmetry assumption, and Bernstein
\cite{Bernstein2020} later gave an elliptic proof of Wang's theorem,
based on the asymptotic structure of almost eigenfunctions of drift
Laplacians on weakly conical ends. Khan
\cite[Theorem~1.1, p.~56]{Khan2023} extended this uniqueness to properly
immersed self-shrinkers of higher codimension, also via Bernstein's
method. Conical self-shrinkers also enter the
gluing constructions of Kapouleas--Kleene--M\o ller
\cite{KapouleasKleeneMoller2018} and the analysis of asymptotically
conical tangent flows by Chodosh--Schulze \cite{ChodoshSchulze2021}.

For fully nonlinear speeds, most of the noncompact theory concerns convex
solutions, Gauss curvature flows, or translating solitons. For positive
powers $K^\alpha$ of the Gauss curvature, Urbas showed that every complete
noncompact convex contracting self-similar solution is a hyperplane
\cite[Theorem~2.1, p.~255]{Urbas1998}, and that convex expanding
solutions exist inside, and are asymptotic to the boundary of, every
convex cone containing no line, possibly in a generalized sense
\cite[Theorem~1, p.~252]{Urbas1998};
rotational translators for
general extrinsic flows were studied by Rengaswami \cite{Rengaswami2021}.
Gao, Li and Ma proved that, for $1\le k\le n-1$ and $\alpha\ge1/k$,
strictly convex closed self-similar solutions of the
$\sigma_k^\alpha$-flow are round spheres
\cite[Theorem~1.4, p.~2]{GaoLiMa2018}.
Guo constructed rotational conical shrinkers for elliptic symmetric speeds
that are homogeneous of degree one \cite[Theorem~1.1, p.~7354]{Guo2017}, and proved their
uniqueness under additional hypotheses on the speed and on the asymptotic
cone \cite[Theorem~2.1, p.~5]{Guo2018}. These results do not cover \eqref{eq:shrinker} with
speed $S_r$ when $r>1$: $S_r$ has degree $r$, and replacing it by
$S_r^{1/r}$ changes both the elliptic equation and its scaling, even on
the same admissibility cone. Rigidity and gap phenomena for
$S_r$-shrinkers were studied in
\cite{AlencarBatistaSilvaNeto2023,BatistaXavier2024,AlencarBessaSilvaNeto2025},
and translating solitons for higher-order mean curvature flows satisfy a
different soliton equation \cite{LiraFarias2026}. The present paper
addresses the complementary problems of existence and uniqueness of
asymptotically conical $S_r$-shrinking ends.

\subsection*{Uniqueness at infinity}
Our main result is a uniqueness theorem at infinity, without any symmetry
assumption. Let $C$ be a regular cone over a smooth closed hypersurface
$\Gamma\subset\Sph^n$, with unit normal $N_C$. We say that an end is
\emph{smoothly asymptotic to $C$ with multiplicity one} if, outside a
compact set, it is the normal graph
\begin{equation}\label{eq:nr-graph-intro}
X_f(\rho,\theta)=\rho\theta+f(\rho,\theta)N_C(\theta),\qquad
|\nabla_C^jf|=o(\rho^{1-j})\quad(j\ge0),
\end{equation}
of a function $f$ over the part $\{\rho\ge\rho_0\}$ of $C$; see
Section~\ref{sec:nonrotational}. This is the hypothesis of Wang
\cite{Wang2014}: the rescalings $tX$ converge to $C$ in
$C^\infty_{\mathrm{loc}}$ as $t\downarrow0$. Let $O(n)$ act on
$\R^{n+1}=\R\times\R^n$ on the second factor and, for $\sigma>0$, let
\[
\calC_\sigma=\{(x,y\omega):x>0,\ y=\sigma x,\ \omega\in\Sph^{n-1}\}
\]
be the rotational cone of slope $\sigma$.

\begin{maintheorem}\label{thm:nr-circular}
Let $1\leq r\leq n-1$, $c<0$, and $\sigma>0$. Suppose that two
smooth ends satisfy \eqref{eq:shrinker} and are smoothly asymptotic,
with multiplicity one, to $\calC_\sigma$, with the orientation converging
to
\(
 N_C=(\sigma,-\omega)/\sqrt{1+\sigma^2}.
\)
Then the two ends coincide outside a compact set.
\end{maintheorem}

No admissibility assumption is made on the ends in
Theorem~\ref{thm:nr-circular}: the principal curvatures of $\calC_\sigma$
lie in the interior of $\Ga_r$ because $r\le n-1$, so smooth convergence
to the cone makes both ends strictly $r$-admissible outside a compact set.

The proof has two steps. First, the equation forces the difference of two
such ends to vanish at infinity faster than every power of the distance
(Lemma~\ref{lem:nr-decay}); flatness alone, however, does not imply
equality. Second, a nonlinear radial change of variables, dictated by the
degree-$r$ homogeneity of $S_r$, together with the Newton tensor of the
cone, produces an auxiliary conical metric on which the difference
satisfies an almost-eigenfunction inequality for a drift Laplacian. The
asymptotic analysis of Bernstein \cite[Theorem~1.2, p.~1898]{Bernstein2020} then
shows that a solution of such an inequality with vanishing asymptotic
trace vanishes identically outside a compact set. For $r=1$ the
construction reduces to Bernstein's proof of Wang's theorem; for $r>1$
the effective metric and the radial drift must be normalized together.
Proposition~\ref{prop:nr-constant-curvatures} extends the argument to
strictly $r$-admissible cones whose links have constant principal
curvatures, including noncircular examples.

The elliptic argument settles the circular case and the case of links
with constant principal curvatures; uniqueness over general admissible
links remains open (Question~\ref{ques:general-cones}).

\subsection*{The rotational profile}
Theorem~\ref{thm:nr-circular} is a uniqueness statement. We now construct
the ends to which it applies, and describe them in detail. For a
positive function $u$ of $x$ we consider the rotational graph
\[
X(x,\omega)=(x,u(x)\omega),\qquad \omega\in\Sph^{n-1},
\]
oriented by $N=(u',-\omega)/\sqrt{1+(u')^2}$. Put
\[
\Psi=xu'-u,\qquad m=u\sqrt{1+(u')^2},\qquad
\kappa=\frac{-c}{\binomc}>0.
\]
In Section~\ref{sec:rotational} we show that \eqref{eq:shrinker} becomes
the profile equation
\begin{equation}\label{eq:profile-intro}
u''=\left(\frac{n-r}{ru}+\kappa m^{r-1}\Psi\right)
       \bigl(1+(u')^2\bigr).
\end{equation}
The quantity $\Psi$ is the angular defect of the profile: $(u/x)'=\Psi/x^2$,
and $\ip{X}{N}=\Psi/\sqrt{1+(u')^2}$.

\begin{maintheorem}\label{thm:main}
Let $1\le r\le n-1$, $c<0$, and $\sigma>0$. There exists a smooth function
$u_\sigma:(0,\infty)\to(0,\infty)$ solving \eqref{eq:profile-intro} such
that $u_\sigma(x)/x\to\sigma$ as $x\to\infty$, and it is unique in the
following strong sense: every positive solution of
\eqref{eq:profile-intro} on a half-line $(d,\infty)$ with $u(x)/x\to\sigma$
coincides with $u_\sigma$ on the common interval of definition. Its
rotational graph $\Sigma_\sigma$ is a smooth embedded $r$-admissible
solution of \eqref{eq:shrinker}, asymptotic to $\calC_\sigma$, and for
every $x>0$,
\begin{align}
0&< u_\sigma(x)-\sigma x
 \le \frac{n-r}{2r\kappa\sigma^r}\,x^{-r},\label{eq:main0}\\
|u_\sigma'(x)-\sigma|
&\le \frac{3(n-r)}{2r\kappa\sigma^r}\,x^{-(r+1)},\label{eq:main1}\\
0&>xu_\sigma'(x)-u_\sigma(x)
 \ge-\frac{n-r}{r\kappa\sigma^r}\,x^{-r}.\label{eq:mainPsi}
\end{align}
In particular, $u_\sigma(x)/x$ decreases strictly to $\sigma$ and
$|A_{\Sigma_\sigma}|=O(|X|^{-1})$ as $x\to\infty$. For every $b>0$ the
tail $\Sigma_\sigma\cap\{x\ge b\}$ is proper and complete in the direction
$x\to\infty$.
\end{maintheorem}

The admissibility assertion deserves emphasis. The principal curvatures of
$\Sigma_\sigma$ are
\[
\mu=-\frac{u_\sigma''}{(1+(u_\sigma')^2)^{3/2}},
\qquad
\lambda=\frac{1}{u_\sigma\sqrt{1+(u_\sigma')^2}},
\]
where $\lambda$ has multiplicity $n-1$. The $r$-admissibility condition is
equivalent to $\mu> -\frac{n-r}{r}\lambda$, and the profile equation shows
that this inequality is equivalent to $\Psi<0$
(Proposition~\ref{prop:admissibility}). Hence admissibility holds on the
whole end for every slope; no smallness assumption on $\sigma$ is needed.
The range $r\le n-1$ is sharp within the strictly admissible rotational
graphical class (Proposition~\ref{prop:gauss-endpoint}).

The profile equation \eqref{eq:profile-intro} is regular at $x=0$: its
right-hand side is real analytic on $\{u>0\}$ for every $x\in\R$. Thus
$u_\sigma$ continues to a maximal interval $(x_\sigma,\infty)$, with
$x_\sigma\in[-\infty,0]$, on which the generating curve remains a graph
over the axis. Theorem~\ref{thm:main} concerns the part of this interval
where $x>0$; Proposition~\ref{prop:maximal-graph} records what can be said
about the maximal graphical interval, and Question~\ref{ques:finite-end}
asks for its description.

\begin{remark}\label{cor:r1unique}
For $r=1$ we have $\kappa=-c$, and Theorem~\ref{thm:main} produces a
positive rotational graphical self-shrinking end with a prescribed conical
asymptote. After matching the normalization of the shrinker coefficient,
Kleene--M\o ller's classification of rotational self-shrinking ends
 \cite[Theorem~3, p.~3947]{KleeneMoller2014} identifies it with their trumpet. Among smooth properly embedded
self-shrinking ends that are asymptotic to $\calC_\sigma$ in the sense of
Wang, uniqueness outside a compact set is Wang's theorem
 \cite[Theorem~1.3, p.~614]{Wang2014}; it is also recovered by Theorem~\ref{thm:nr-circular}
and Corollary~\ref{cor:classification}, whose proof reduces, for $r=1$, to Bernstein's argument
\cite{Bernstein2020}.
\end{remark}

The next result describes the complete asymptotic structure of the
profile. For $r=1$ it is consistent with the ODE analysis of
Kleene--M\o ller \cite{KleeneMoller2014}.

\begin{maintheorem}\label{thm:expansion-intro}
Let $u_\sigma$ be the profile in Theorem~\ref{thm:main}. There is a
uniquely determined sequence $(c_j)_{j\ge1}$ such that, for all integers
$N\ge1$ and $k\ge0$,
\begin{equation}\label{eq:expansion-intro}
\frac{d^k}{dx^k}\left(u_\sigma(x)-\sigma x
-\sum_{j=1}^N c_jx^{1-j(r+1)}\right)
=O\left(x^{1-(N+1)(r+1)-k}\right)
\qquad(x\to\infty).
\end{equation}
The leading coefficient is
\begin{equation}\label{eq:first-coefficient}
c_1=\frac{n-r}{r(r+1)\kappa\sigma^r(1+\sigma^2)^{(r-1)/2}},
\end{equation}
and the meridian curvature satisfies
\begin{equation}\label{eq:meridian-expansion}
\mu(x)=-\frac{n-r}{\kappa\sigma^r(1+\sigma^2)^{(r+2)/2}}
 x^{-r-2}+O(x^{-2r-3}).
\end{equation}
\end{maintheorem}

Combining Theorems~\ref{thm:nr-circular}, \ref{thm:main}
and~\ref{thm:expansion-intro} we obtain a complete description of the
ends asymptotic to a circular cone.

\begin{corollary}\label{cor:classification}
Let $1\le r\le n-1$, $c<0$, and $\sigma>0$. Every smooth end satisfying
\eqref{eq:shrinker} that is smoothly asymptotic, with multiplicity one,
to $\calC_\sigma$, with the orientation converging to
$N_C=(\sigma,-\omega)/\sqrt{1+\sigma^2}$, coincides outside a compact set
with the rotational end $\Sigma_\sigma$ of Theorem~\ref{thm:main}. In
particular, it is rotationally symmetric and $r$-admissible outside a
compact set, and it has the asymptotic expansion of
Theorem~\ref{thm:expansion-intro}.
\end{corollary}

\begin{proof}
By Theorem~\ref{thm:expansion-intro}, the end $\Sigma_\sigma$ is smoothly
asymptotic to $\calC_\sigma$ with multiplicity one
(Proposition~\ref{prop:geometry}), and its orientation converges to
$N_C$. Theorem~\ref{thm:nr-circular} applies to $\Sigma_\sigma$ and the
given end.
\end{proof}

The asymptotic hypothesis of Theorem~\ref{thm:nr-circular} is stronger
than the scalar slope condition $u(x)/x\to\sigma$ that suffices for
uniqueness within the rotational class (Theorem~\ref{thm:main}); in
exchange, no symmetry is assumed.

\subsection*{Dependence on the slope}
The rotational profiles depend real analytically on the cone slope.

\begin{maintheorem}\label{thm:analytic-intro}
For every compact interval $I\Subset(0,\infty)$ there is $R_I>0$ such
that $\sigma\mapsto u_\sigma-\sigma x$ is real analytic on a neighborhood
of $I$ with values in weighted H\"older spaces on $[R_I,\infty)$ that
control $u_\sigma-\sigma x$ and all its derivatives at infinity, and
$\sigma\mapsto u_\sigma\in C^{k,\alpha}([\varepsilon,E])$ is real
analytic for every $k\ge0$, $0<\alpha<1$, and $0<\varepsilon<E$. Moreover,
$\partial_\sigma u_\sigma(x)=x+O_I(x^{-r})$, and for slopes in the
interior of $I$ the ends $\Sigma_\sigma\cap\{x>R_I\}$ form a
real-analytic foliation of their image.
\end{maintheorem}

Theorem~\ref{thm:analytic-intro} is proved in
Section~\ref{sec:analytic} (Theorem~\ref{thm:analytic-slope-dependence}
and Corollary~\ref{cor:analytic-exterior-foliation}). It does not imply a
global ordering of the family: numerical computations indicate that
profiles with different slopes intersect near $x=0$
(Remark~\ref{rem:numerical-crossing}).

Figure~\ref{fig:rotational-profiles} illustrates the rotational profiles
and their asymptotic cones for $r=1,2$. This figure and
Remarks~\ref{rem:numerical-maximal} and~\ref{rem:numerical-crossing}
report numerical computations; they are illustrations, not certified
bounds, and no proof in the paper uses them.

\begin{figure}[htbp]
\centering
\definecolor{c1}{RGB}{31,119,180}
\definecolor{c2}{RGB}{217,95,2}
\definecolor{c3}{RGB}{27,158,119}
\begin{tikzpicture}
\begin{groupplot}[
  group style={group size=2 by 2, horizontal sep=1.1cm, vertical sep=1.3cm},
  width=0.5\textwidth, height=0.36\textwidth,
  tick label style={font=\scriptsize}, label style={font=\footnotesize},
  title style={font=\footnotesize}, every axis plot/.append style={line width=0.8pt},
  grid=major, grid style={gray!20}, axis line style={gray!70}]
\nextgroupplot[title={$r=1$}, xmin=-5, xmax=25, ymin=0, ymax=52,
  ylabel={$u_\sigma(x)$},
  legend style={font=\scriptsize, at={(0.03,0.97)}, anchor=north west, draw=none, fill=white, fill opacity=0.8, text opacity=1},
  legend cell align=left]
\addplot[c1,thick] coordinates {(-2.578,2.347) (-2.375,1.945) (-2.079,1.731) (-1.768,1.587) (-1.453,1.482) (-1.134,1.408) (-0.814,1.357) (-0.493,1.329) (-0.172,1.322) (0.149,1.335) (0.470,1.365) (0.791,1.412) (1.110,1.473) (1.428,1.547) (1.746,1.632) (2.063,1.726) (2.379,1.828) (2.694,1.937) (3.008,2.051) (3.322,2.170) (3.636,2.294) (3.949,2.421) (4.261,2.551) (4.573,2.683) (4.885,2.818) (5.197,2.954) (5.508,3.092) (5.819,3.232) (6.130,3.373) (6.441,3.515) (6.751,3.658) (7.062,3.802) (7.372,3.946) (7.682,4.092) (7.993,4.238) (8.303,4.384) (8.613,4.532) (8.922,4.679) (9.232,4.827) (9.542,4.975) (9.852,5.124) (10.161,5.273) (10.471,5.423) (10.781,5.572) (11.090,5.722) (11.400,5.872) (11.709,6.023) (12.018,6.173) (12.328,6.324) (12.637,6.475) (12.947,6.626) (13.256,6.777) (13.565,6.928) (13.874,7.080) (14.184,7.231) (14.493,7.383) (14.802,7.535) (15.111,7.687) (15.420,7.839) (15.730,7.991) (16.039,8.143) (16.348,8.295) (16.657,8.448) (16.966,8.600) (17.275,8.753) (17.584,8.905) (17.893,9.058) (18.202,9.210) (18.511,9.363) (18.820,9.516) (19.129,9.669) (19.439,9.822) (19.748,9.975) (20.057,10.128) (20.366,10.281) (20.675,10.434) (20.984,10.587) (21.293,10.740) (21.602,10.893) (21.911,11.046) (22.219,11.199) (22.528,11.353) (22.837,11.506) (23.146,11.659) (23.455,11.813) (23.764,11.966) (24.073,12.119) (24.382,12.273) (24.691,12.426) (25.000,12.580)};
\addlegendentry{$\sigma=0.5$}
\addplot[c1,dashed,forget plot,domain=0:25,samples=2] {0.5*x};
\addplot[c2,thick] coordinates {(-1.678,2.080) (-1.492,1.605) (-1.189,1.335) (-0.858,1.186) (-0.518,1.129) (-0.176,1.150) (0.163,1.237) (0.495,1.377) (0.821,1.557) (1.142,1.766) (1.457,1.996) (1.769,2.241) (2.078,2.496) (2.385,2.760) (2.690,3.029) (2.993,3.303) (3.296,3.581) (3.597,3.861) (3.898,4.143) (4.199,4.428) (4.499,4.713) (4.798,5.000) (5.098,5.289) (5.397,5.577) (5.695,5.867) (5.994,6.157) (6.292,6.448) (6.590,6.740) (6.888,7.031) (7.186,7.324) (7.484,7.616) (7.782,7.909) (8.080,8.202) (8.377,8.495) (8.675,8.789) (8.972,9.082) (9.269,9.376) (9.567,9.670) (9.864,9.965) (10.161,10.259) (10.458,10.553) (10.755,10.848) (11.053,11.142) (11.350,11.437) (11.647,11.732) (11.944,12.027) (12.241,12.322) (12.538,12.617) (12.835,12.912) (13.131,13.207) (13.428,13.503) (13.725,13.798) (14.022,14.093) (14.319,14.389) (14.616,14.684) (14.913,14.979) (15.209,15.275) (15.506,15.571) (15.803,15.866) (16.100,16.162) (16.397,16.457) (16.693,16.753) (16.990,17.049) (17.287,17.345) (17.584,17.640) (17.880,17.936) (18.177,18.232) (18.474,18.528) (18.770,18.824) (19.067,19.119) (19.364,19.415) (19.660,19.711) (19.957,20.007) (20.254,20.303) (20.550,20.599) (20.847,20.895) (21.144,21.191) (21.440,21.487) (21.737,21.783) (22.034,22.079) (22.330,22.375) (22.627,22.671) (22.924,22.967) (23.220,23.263) (23.517,23.559) (23.814,23.856) (24.110,24.152) (24.407,24.448) (24.703,24.744) (25.000,25.040)};
\addlegendentry{$\sigma=1$}
\addplot[c2,dashed,forget plot,domain=0:25,samples=2] {1.0*x};
\addplot[c3,thick] coordinates {(-1.245,1.574) (-0.990,0.994) (-0.564,0.826) (-0.134,0.968) (0.255,1.323) (0.604,1.788) (0.930,2.303) (1.242,2.841) (1.547,3.392) (1.848,3.950) (2.145,4.514) (2.441,5.080) (2.735,5.648) (3.029,6.218) (3.321,6.790) (3.613,7.362) (3.904,7.935) (4.195,8.508) (4.486,9.082) (4.776,9.656) (5.067,10.231) (5.357,10.806) (5.647,11.381) (5.936,11.956) (6.226,12.532) (6.516,13.107) (6.805,13.683) (7.094,14.259) (7.384,14.835) (7.673,15.411) (7.962,15.987) (8.251,16.563) (8.541,17.140) (8.830,17.716) (9.119,18.292) (9.408,18.869) (9.697,19.445) (9.986,20.022) (10.275,20.598) (10.564,21.175) (10.853,21.751) (11.142,22.328) (11.431,22.905) (11.719,23.481) (12.008,24.058) (12.297,24.635) (12.586,25.212) (12.875,25.788) (13.164,26.365) (13.452,26.942) (13.741,27.519) (14.030,28.095) (14.319,28.672) (14.607,29.249) (14.896,29.826) (15.185,30.403) (15.474,30.980) (15.762,31.557) (16.051,32.134) (16.340,32.710) (16.629,33.287) (16.917,33.864) (17.206,34.441) (17.495,35.018) (17.783,35.595) (18.072,36.172) (18.361,36.749) (18.650,37.326) (18.938,37.903) (19.227,38.480) (19.516,39.057) (19.804,39.634) (20.093,40.211) (20.382,40.788) (20.670,41.365) (20.959,41.942) (21.248,42.519) (21.536,43.096) (21.825,43.673) (22.114,44.250) (22.402,44.827) (22.691,45.404) (22.980,45.981) (23.268,46.558) (23.557,47.135) (23.845,47.712) (24.134,48.289) (24.423,48.866) (24.711,49.443) (25.000,50.020)};
\addlegendentry{$\sigma=2$}
\addplot[c3,dashed,forget plot,domain=0:25,samples=2] {2.0*x};
\draw[gray!70] (axis cs:0,0) -- (axis cs:0,52);
\nextgroupplot[title={$r=2$}, xmin=-5, xmax=25, ymin=0, ymax=52]
\addplot[c1,thick] coordinates {(-2.807,0.915) (-2.518,1.125) (-2.195,1.135) (-1.873,1.078) (-1.553,0.993) (-1.234,0.901) (-0.915,0.813) (-0.594,0.739) (-0.272,0.687) (0.052,0.664) (0.375,0.673) (0.698,0.714) (1.019,0.783) (1.338,0.876) (1.656,0.986) (1.971,1.109) (2.286,1.241) (2.600,1.379) (2.913,1.521) (3.225,1.666) (3.537,1.814) (3.849,1.963) (4.161,2.114) (4.473,2.265) (4.784,2.418) (5.096,2.570) (5.407,2.724) (5.718,2.877) (6.030,3.031) (6.341,3.185) (6.652,3.339) (6.963,3.494) (7.274,3.648) (7.585,3.803) (7.896,3.958) (8.207,4.112) (8.518,4.267) (8.829,4.422) (9.140,4.577) (9.451,4.732) (9.762,4.887) (10.074,5.043) (10.385,5.198) (10.696,5.353) (11.007,5.508) (11.318,5.663) (11.628,5.819) (11.939,5.974) (12.250,6.129) (12.561,6.285) (12.872,6.440) (13.183,6.595) (13.494,6.750) (13.805,6.906) (14.116,7.061) (14.427,7.217) (14.738,7.372) (15.049,7.527) (15.360,7.683) (15.671,7.838) (15.982,7.993) (16.293,8.149) (16.604,8.304) (16.915,8.460) (17.226,8.615) (17.537,8.770) (17.848,8.926) (18.159,9.081) (18.470,9.237) (18.781,9.392) (19.092,9.548) (19.403,9.703) (19.714,9.858) (20.025,10.014) (20.336,10.169) (20.647,10.325) (20.958,10.480) (21.269,10.636) (21.579,10.791) (21.890,10.946) (22.201,11.102) (22.512,11.257) (22.823,11.413) (23.134,11.568) (23.445,11.724) (23.756,11.879) (24.067,12.035) (24.378,12.190) (24.689,12.345) (25.000,12.501)};
\addplot[c1,dashed,domain=0:25,samples=2] {0.5*x};
\addplot[c2,thick] coordinates {(-1.366,1.154) (-1.111,0.782) (-0.798,0.557) (-0.466,0.436) (-0.127,0.429) (0.207,0.532) (0.529,0.719) (0.840,0.955) (1.144,1.218) (1.444,1.495) (1.742,1.778) (2.039,2.066) (2.334,2.355) (2.630,2.646) (2.925,2.938) (3.220,3.231) (3.514,3.524) (3.809,3.817) (4.104,4.111) (4.398,4.404) (4.693,4.698) (4.987,4.992) (5.282,5.286) (5.576,5.580) (5.870,5.874) (6.165,6.168) (6.459,6.462) (6.753,6.756) (7.048,7.050) (7.342,7.344) (7.636,7.638) (7.931,7.933) (8.225,8.227) (8.519,8.521) (8.814,8.815) (9.108,9.109) (9.402,9.404) (9.697,9.698) (9.991,9.992) (10.285,10.286) (10.580,10.581) (10.874,10.875) (11.168,11.169) (11.462,11.463) (11.757,11.758) (12.051,12.052) (12.345,12.346) (12.640,12.640) (12.934,12.935) (13.228,13.229) (13.523,13.523) (13.817,13.817) (14.111,14.112) (14.405,14.406) (14.700,14.700) (14.994,14.995) (15.288,15.289) (15.583,15.583) (15.877,15.877) (16.171,16.172) (16.466,16.466) (16.760,16.760) (17.054,17.055) (17.348,17.349) (17.643,17.643) (17.937,17.937) (18.231,18.232) (18.526,18.526) (18.820,18.820) (19.114,19.114) (19.408,19.409) (19.703,19.703) (19.997,19.997) (20.291,20.292) (20.586,20.586) (20.880,20.880) (21.174,21.174) (21.468,21.469) (21.763,21.763) (22.057,22.057) (22.351,22.352) (22.646,22.646) (22.940,22.940) (23.234,23.234) (23.529,23.529) (23.823,23.823) (24.117,24.117) (24.411,24.412) (24.706,24.706) (25.000,25.000)};
\addplot[c2,dashed,domain=0:25,samples=2] {1.0*x};
\addplot[c3,thick] coordinates {(-0.882,0.777) (-0.574,0.263) (-0.147,0.203) (0.226,0.595) (0.539,1.129) (0.834,1.693) (1.125,2.263) (1.414,2.836) (1.702,3.410) (1.990,3.985) (2.278,4.560) (2.566,5.134) (2.854,5.709) (3.141,6.284) (3.429,6.860) (3.717,7.435) (4.004,8.010) (4.292,8.585) (4.580,9.160) (4.867,9.735) (5.155,10.310) (5.443,10.886) (5.730,11.461) (6.018,12.036) (6.305,12.611) (6.593,13.186) (6.881,13.762) (7.168,14.337) (7.456,14.912) (7.743,15.487) (8.031,16.062) (8.319,16.638) (8.606,17.213) (8.894,17.788) (9.182,18.363) (9.469,18.939) (9.757,19.514) (10.044,20.089) (10.332,20.664) (10.620,21.239) (10.907,21.815) (11.195,22.390) (11.482,22.965) (11.770,23.540) (12.058,24.115) (12.345,24.691) (12.633,25.266) (12.920,25.841) (13.208,26.416) (13.496,26.991) (13.783,27.567) (14.071,28.142) (14.359,28.717) (14.646,29.292) (14.934,29.868) (15.221,30.443) (15.509,31.018) (15.797,31.593) (16.084,32.168) (16.372,32.744) (16.659,33.319) (16.947,33.894) (17.235,34.469) (17.522,35.044) (17.810,35.620) (18.097,36.195) (18.385,36.770) (18.673,37.345) (18.960,37.921) (19.248,38.496) (19.535,39.071) (19.823,39.646) (20.111,40.221) (20.398,40.797) (20.686,41.372) (20.973,41.947) (21.261,42.522) (21.549,43.097) (21.836,43.673) (22.124,44.248) (22.412,44.823) (22.699,45.398) (22.987,45.974) (23.274,46.549) (23.562,47.124) (23.850,47.699) (24.137,48.274) (24.425,48.850) (24.712,49.425) (25.000,50.000)};
\addplot[c3,dashed,domain=0:25,samples=2] {2.0*x};
\draw[gray!70] (axis cs:0,0) -- (axis cs:0,52);
\nextgroupplot[xmin=-2.5, xmax=2.5, ymin=0, ymax=3, xlabel={$x$}, ylabel={$u_\sigma(x)$}]
\addplot[c1,thick] coordinates {(-2.500,2.099) (-2.477,2.063) (-2.450,2.028) (-2.422,1.994) (-2.391,1.961) (-2.357,1.928) (-2.322,1.896) (-2.284,1.866) (-2.244,1.836) (-2.203,1.807) (-2.159,1.778) (-2.115,1.751) (-2.068,1.725) (-2.021,1.699) (-1.972,1.675) (-1.922,1.651) (-1.870,1.628) (-1.818,1.606) (-1.765,1.585) (-1.711,1.565) (-1.656,1.546) (-1.600,1.527) (-1.544,1.509) (-1.487,1.492) (-1.429,1.476) (-1.370,1.461) (-1.311,1.446) (-1.252,1.432) (-1.192,1.419) (-1.131,1.407) (-1.071,1.396) (-1.009,1.385) (-0.948,1.375) (-0.886,1.366) (-0.823,1.358) (-0.761,1.351) (-0.698,1.345) (-0.635,1.339) (-0.571,1.334) (-0.508,1.330) (-0.444,1.327) (-0.381,1.325) (-0.317,1.323) (-0.253,1.322) (-0.189,1.322) (-0.126,1.323) (-0.062,1.324) (0.002,1.327) (0.065,1.330) (0.129,1.333) (0.192,1.338) (0.255,1.343) (0.318,1.349) (0.381,1.355) (0.444,1.362) (0.507,1.370) (0.569,1.378) (0.631,1.387) (0.693,1.396) (0.754,1.406) (0.816,1.416) (0.877,1.427) (0.938,1.438) (0.999,1.450) (1.059,1.462) (1.119,1.475) (1.179,1.488) (1.239,1.501) (1.298,1.515) (1.358,1.529) (1.417,1.544) (1.476,1.559) (1.534,1.574) (1.593,1.589) (1.651,1.605) (1.709,1.621) (1.766,1.637) (1.824,1.654) (1.881,1.671) (1.938,1.688) (1.995,1.705) (2.052,1.722) (2.109,1.740) (2.165,1.758) (2.221,1.776) (2.277,1.794) (2.333,1.813) (2.389,1.831) (2.445,1.850) (2.500,1.869)};
\addplot[c1,dashed,domain=0:2.500,samples=2] {0.5*x};
\addplot[c2,thick] coordinates {(-1.678,2.080) (-1.676,2.037) (-1.671,1.993) (-1.662,1.950) (-1.652,1.907) (-1.638,1.864) (-1.622,1.822) (-1.603,1.780) (-1.581,1.739) (-1.557,1.698) (-1.531,1.657) (-1.502,1.617) (-1.470,1.578) (-1.436,1.540) (-1.400,1.502) (-1.361,1.465) (-1.320,1.430) (-1.276,1.395) (-1.229,1.362) (-1.180,1.330) (-1.128,1.300) (-1.073,1.271) (-1.016,1.244) (-0.956,1.220) (-0.894,1.197) (-0.829,1.178) (-0.762,1.161) (-0.693,1.148) (-0.623,1.138) (-0.552,1.131) (-0.479,1.127) (-0.407,1.127) (-0.335,1.131) (-0.263,1.137) (-0.192,1.147) (-0.123,1.159) (-0.054,1.174) (0.012,1.191) (0.078,1.209) (0.142,1.230) (0.204,1.252) (0.266,1.275) (0.326,1.299) (0.384,1.325) (0.442,1.351) (0.498,1.378) (0.554,1.406) (0.608,1.435) (0.662,1.464) (0.715,1.494) (0.767,1.524) (0.818,1.555) (0.869,1.586) (0.919,1.618) (0.968,1.649) (1.017,1.682) (1.066,1.714) (1.113,1.747) (1.161,1.779) (1.208,1.813) (1.254,1.846) (1.301,1.879) (1.346,1.913) (1.392,1.947) (1.437,1.981) (1.482,2.015) (1.527,2.049) (1.571,2.084) (1.615,2.118) (1.659,2.153) (1.703,2.188) (1.746,2.222) (1.789,2.257) (1.832,2.292) (1.875,2.327) (1.918,2.363) (1.960,2.398) (2.003,2.433) (2.045,2.468) (2.087,2.504) (2.129,2.539) (2.170,2.575) (2.212,2.611) (2.253,2.646) (2.295,2.682) (2.336,2.718) (2.377,2.753) (2.418,2.789) (2.459,2.825) (2.500,2.861)};
\addplot[c2,dashed,domain=0:2.500,samples=2] {1.0*x};
\addplot[c3,thick] coordinates {(-1.245,1.574) (-1.241,1.509) (-1.233,1.445) (-1.219,1.381) (-1.199,1.317) (-1.175,1.254) (-1.144,1.192) (-1.107,1.131) (-1.064,1.072) (-1.012,1.015) (-0.952,0.962) (-0.881,0.913) (-0.799,0.872) (-0.704,0.841) (-0.599,0.827) (-0.492,0.830) (-0.389,0.850) (-0.296,0.882) (-0.211,0.922) (-0.134,0.967) (-0.064,1.016) (0.001,1.068) (0.062,1.121) (0.119,1.176) (0.173,1.232) (0.225,1.289) (0.275,1.346) (0.323,1.404) (0.369,1.462) (0.415,1.521) (0.458,1.580) (0.501,1.640) (0.543,1.699) (0.585,1.759) (0.625,1.819) (0.665,1.879) (0.704,1.939) (0.743,2.000) (0.781,2.060) (0.819,2.121) (0.856,2.181) (0.893,2.242) (0.930,2.303) (0.966,2.364) (1.002,2.425) (1.038,2.486) (1.074,2.547) (1.109,2.608) (1.144,2.669) (1.179,2.730) (1.214,2.792) (1.249,2.853) (1.283,2.914) (1.318,2.975) (1.352,3.037)};
\addplot[c3,dashed,domain=0:1.500,samples=2] {2.0*x};
\draw[gray!70] (axis cs:0,0) -- (axis cs:0,3);
\nextgroupplot[xmin=-2.5, xmax=2.5, ymin=0, ymax=3, xlabel={$x$}]
\addplot[c1,thick] coordinates {(-2.500,1.128) (-2.441,1.137) (-2.380,1.141) (-2.319,1.142) (-2.258,1.140) (-2.197,1.135) (-2.137,1.128) (-2.078,1.119) (-2.019,1.109) (-1.961,1.097) (-1.903,1.085) (-1.846,1.071) (-1.790,1.057) (-1.733,1.043) (-1.678,1.028) (-1.622,1.013) (-1.567,0.997) (-1.511,0.981) (-1.456,0.966) (-1.401,0.950) (-1.346,0.934) (-1.291,0.918) (-1.236,0.902) (-1.180,0.886) (-1.125,0.870) (-1.070,0.855) (-1.014,0.840) (-0.958,0.825) (-0.902,0.810) (-0.845,0.796) (-0.789,0.782) (-0.732,0.769) (-0.674,0.756) (-0.617,0.744) (-0.559,0.732) (-0.500,0.722) (-0.441,0.711) (-0.382,0.702) (-0.323,0.694) (-0.263,0.686) (-0.202,0.680) (-0.142,0.674) (-0.081,0.670) (-0.020,0.666) (0.041,0.664) (0.102,0.663) (0.163,0.663) (0.224,0.664) (0.285,0.667) (0.346,0.670) (0.407,0.675) (0.467,0.681) (0.527,0.688) (0.587,0.696) (0.646,0.705) (0.705,0.715) (0.764,0.726) (0.822,0.737) (0.879,0.750) (0.937,0.763) (0.993,0.777) (1.049,0.791) (1.105,0.806) (1.161,0.822) (1.216,0.838) (1.270,0.855) (1.324,0.872) (1.378,0.889) (1.432,0.907) (1.485,0.925) (1.538,0.944) (1.590,0.962) (1.642,0.981) (1.694,1.001) (1.746,1.020) (1.798,1.040) (1.849,1.060) (1.900,1.080) (1.951,1.101) (2.002,1.121) (2.052,1.142) (2.102,1.163) (2.153,1.184) (2.203,1.205) (2.253,1.226) (2.302,1.248) (2.352,1.269) (2.401,1.291) (2.451,1.313) (2.500,1.334)};
\addplot[c1,dashed,domain=0:2.500,samples=2] {0.5*x};
\addplot[c2,thick] coordinates {(-1.366,1.154) (-1.357,1.112) (-1.340,1.071) (-1.318,1.030) (-1.293,0.990) (-1.265,0.951) (-1.235,0.913) (-1.202,0.875) (-1.168,0.838) (-1.131,0.801) (-1.092,0.765) (-1.052,0.730) (-1.010,0.696) (-0.965,0.663) (-0.919,0.630) (-0.870,0.599) (-0.820,0.569) (-0.767,0.541) (-0.711,0.515) (-0.653,0.490) (-0.592,0.468) (-0.528,0.450) (-0.462,0.435) (-0.393,0.424) (-0.323,0.418) (-0.252,0.417) (-0.181,0.422) (-0.112,0.431) (-0.044,0.444) (0.021,0.461) (0.083,0.482) (0.143,0.504) (0.201,0.529) (0.257,0.556) (0.310,0.584) (0.362,0.613) (0.412,0.643) (0.461,0.674) (0.509,0.705) (0.556,0.737) (0.602,0.770) (0.647,0.803) (0.691,0.836) (0.734,0.870) (0.777,0.904) (0.819,0.938) (0.861,0.973) (0.903,1.007) (0.944,1.042) (0.984,1.077) (1.025,1.112) (1.065,1.147) (1.104,1.183) (1.144,1.218) (1.183,1.253) (1.222,1.289) (1.261,1.325) (1.300,1.360) (1.339,1.396) (1.377,1.432) (1.416,1.468) (1.454,1.504) (1.492,1.540) (1.530,1.576) (1.568,1.612) (1.606,1.648) (1.644,1.684) (1.681,1.720) (1.719,1.756) (1.756,1.792) (1.794,1.828) (1.831,1.865) (1.869,1.901) (1.906,1.937) (1.944,1.973) (1.981,2.010) (2.018,2.046) (2.055,2.082) (2.093,2.118) (2.130,2.155) (2.167,2.191) (2.204,2.227) (2.241,2.264) (2.278,2.300) (2.315,2.336) (2.352,2.373) (2.389,2.409) (2.426,2.446) (2.463,2.482) (2.500,2.518)};
\addplot[c2,dashed,domain=0:2.500,samples=2] {1.0*x};
\addplot[c3,thick] coordinates {(-0.882,0.777) (-0.869,0.711) (-0.847,0.646) (-0.818,0.582) (-0.784,0.518) (-0.745,0.456) (-0.701,0.395) (-0.650,0.336) (-0.592,0.279) (-0.524,0.226) (-0.441,0.182) (-0.338,0.159) (-0.230,0.170) (-0.139,0.207) (-0.064,0.256) (-0.001,0.311) (0.054,0.369) (0.105,0.428) (0.151,0.489) (0.195,0.550) (0.237,0.612) (0.277,0.674) (0.316,0.736) (0.353,0.799) (0.390,0.862) (0.426,0.925) (0.461,0.988) (0.496,1.051) (0.531,1.115) (0.565,1.178) (0.599,1.241) (0.632,1.305) (0.666,1.368) (0.699,1.432) (0.732,1.496) (0.765,1.559) (0.798,1.623) (0.831,1.686) (0.863,1.750) (0.896,1.814) (0.929,1.878) (0.961,1.941) (0.993,2.005) (1.026,2.069) (1.058,2.132) (1.091,2.196) (1.123,2.260) (1.155,2.324) (1.187,2.387) (1.219,2.451) (1.252,2.515) (1.284,2.579) (1.316,2.642) (1.348,2.706) (1.380,2.770) (1.412,2.834) (1.444,2.898) (1.476,2.961) (1.509,3.025)};
\addplot[c3,dashed,domain=0:1.500,samples=2] {2.0*x};
\draw[gray!70] (axis cs:0,0) -- (axis cs:0,3);
\end{groupplot}
\end{tikzpicture}
\caption{Numerical approximations of the rotational profiles $u_\sigma$
for $n=3$, $\kappa=1$ and $\sigma=0.5,1,2$, for $r=1$ ($c=-1$, left) and
$r=2$ ($c=-2$, right). Solid curves: computed profiles, drawn until
$|u_\sigma'|=60$; dashed lines: the cones $u=\sigma x$, $x>0$. The lower
row enlarges $-2.5\le x\le2.5$. The profiles were computed from terminal
data $u=\sigma x+c_1x^{-r}$, $u'=\sigma-rc_1x^{-r-1}$ at $x=30$ with an
implicit Runge--Kutta (Radau) method, relative tolerance $10^{-11}$.}
\label{fig:rotational-profiles}
\end{figure}
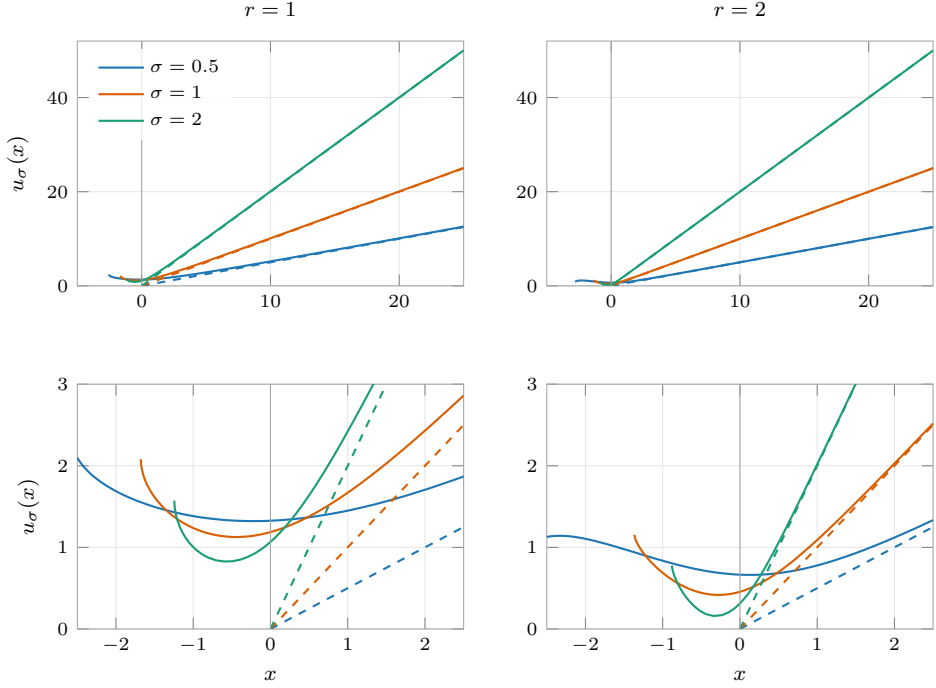

\textbf{Organization.}
Although Theorem~\ref{thm:nr-circular} is stated first, its
proof in Section~\ref{sec:nonrotational} is independent of the rotational
theory, which is developed first because it is elementary and provides the
model ends.
In Section~\ref{sec:rotational} we derive the profile equation, prove the
admissibility criterion, and treat the case $r=n$.
Section~\ref{sec:construction} constructs the profile by finite-radius
approximation, and Section~\ref{sec:uniqueness} proves rotational
uniqueness and an explicit eventual ordering. Section~\ref{sec:expansion}
establishes the complete asymptotic expansion, and
Section~\ref{sec:analytic} proves analytic dependence on the slope and the
exterior foliation. Section~\ref{sec:nonrotational} contains the proof of
Theorem~\ref{thm:nr-circular} and its extension to cones whose links have
constant principal curvatures. Section~\ref{sec:geometry} collects
geometric consequences, including the maximal graphical continuation of
the profile, and Section~\ref{sec:finer} discusses open problems.

\section{Rotational geometry and the profile equation}\label{sec:rotational}

Let $O(n)$ act on $\R^{n+1}=\R\times\R^n$ in the standard way on the second
factor.  A rotational hypersurface is parametrized by
\begin{equation*}
X(t,\omega)=(x(t),y(t)\omega),\qquad \omega\in\Sph^{n-1},
\end{equation*}
where $y>0$.  Put
\[
W=\sqrt{(x')^2+(y')^2}
\]
and choose the unit normal
\begin{equation*}
N=\frac1W(y',-x'\omega).
\end{equation*}
The meridian principal curvature and the transversal principal curvature are
\begin{equation*}
\mu=\frac{y'x''-y''x'}{W^3},
\qquad
\lambda=\frac{x'}{yW},
\end{equation*}
with $\lambda$ of multiplicity $n-1$.  Moreover,
\begin{equation*}
\ip{X}{N}=\frac{xy'-yx'}{W}.
\end{equation*}
Thus
\begin{equation}\label{eq:Sr}
S_r=\binom{n-1}{r-1}\mu\lambda^{r-1}
     +\binom{n-1}{r}\lambda^r.
\end{equation}

We henceforth use a graphical profile
\[
x(t)=t,\qquad y(t)=u(t)>0.
\]
Then $W=\sqrt{1+(u')^2}$ and
\begin{equation}\label{eq:curv-graph}
\mu=-\frac{u''}{W^3},
\qquad
\lambda=\frac1{uW},
\qquad
\ip{X}{N}=\frac{xu'-u}{W}.
\end{equation}
As in the introduction, put
\begin{equation*}
\Psi=xu'-u,\qquad m=uW,
\qquad
\kappa=\frac{-c}{\binom{n-1}{r-1}}>0.
\end{equation*}
Using
\(
\binom{n-1}{r}=\frac{n-r}{r}\binom{n-1}{r-1},
\)
equations \eqref{eq:shrinker} and \eqref{eq:Sr} become
\begin{equation}\label{eq:profile}
u''=\left(\frac{n-r}{ru}+\kappa m^{r-1}\Psi\right)
       \bigl(1+(u')^2\bigr).
\end{equation}
This is the profile equation \eqref{eq:profile-intro} of the introduction.

\begin{proposition}\label{prop:admissibility}
Let $u>0$ solve \eqref{eq:profile}.  At every point at which $\Psi<0$, the
associated rotational hypersurface has a principal-curvature vector in
$\Ga_r$.  Conversely, if $\lambda>0$, $r$-admissibility is equivalent to
$\Psi<0$.
\end{proposition}

\begin{proof}
For $1\le j\le r$, equation \eqref{eq:Sr} with $r$ replaced by $j$ gives
\[
S_j=\binom{n-1}{j-1}\lambda^{j-1}
\left(\mu+\frac{n-j}{j}\lambda\right).
\]
Since $j\mapsto(n-j)/j$ is decreasing and $\lambda>0$, all the inequalities
$S_1>0,\ldots,S_r>0$ are equivalent to
\[
\mu> -\frac{n-r}{r}\lambda.
\]
Set
\[
\Phi=\frac{n-r}{ru}+\kappa m^{r-1}\Psi.
\]
By \eqref{eq:profile} and \eqref{eq:curv-graph},
\[
\mu=-\frac{\Phi}{W},
\qquad
\frac{n-r}{r}\lambda=\frac{n-r}{ruW}.
\]
Hence,
\[
\mu+\frac{n-r}{r}\lambda
=-\frac{\kappa m^{r-1}\Psi}{W},
\]
which is positive exactly when $\Psi<0$.
\end{proof}

The following proposition describes the case $r=n$ and explains the
restriction in Theorem~\ref{thm:main}.

\begin{proposition}\label{prop:gauss-endpoint}
Let $c<0$ and suppose that a positive graphical solution of the rotational
$S_n$-shrinker equation is defined on a right half-line and satisfies
$u(x)/x\to\sigma>0$; set $\kappa=-c>0$.  Then $u(x)=\sigma x$.
Therefore, the associated
cone has $S_n=0$, and its curvature vector lies on the boundary of $\Gamma_n$; in particular, no
strictly $n$-admissible rotational graphical end can have a positive conical
asymptote.
\end{proposition}

\begin{proof}
For $r=n$ the first term in \eqref{eq:profile} vanishes and, with
$W=(1+(u')^2)^{1/2}$ and $m=uW$,
\[
u''=\kappa m^{n-1}W^2\Psi,
\qquad
\Psi'=xu''=\mathcal A(x)\Psi,
\qquad
\mathcal A(x):=\kappa x m^{n-1}W^2>0.
\]
Hence, for any $x\ge x_0$ in the interval of definition,
\[
\Psi(x)=\Psi(x_0)\exp\left(\int_{x_0}^x\mathcal A(s)\,\dd s\right).
\]
Since $u(x)/x\to\sigma$, one has $u(x)\ge(\sigma/2)x$ for all sufficiently
large $x$.  Thus
\[
\mathcal A(x)\ge\kappa x u(x)^{n-1}
\ge\kappa(\sigma/2)^{n-1}x^n.
\]
If $\Psi(x_0)\ne0$, then $\Psi$ has a fixed sign and grows in modulus at
least like $\exp(Cx^{n+1})$.  Because
$(u/x)'=\Psi/x^2$, the integral of $(u/x)'$ on a right half-line would
diverge, contradicting the finite limit $u(x)/x\to\sigma$.  Therefore
$\Psi\equiv0$, whence $(u/x)'=0$ and $u=\sigma x$.  Its meridian curvature
is zero, so $S_n=0$ and the principal-curvature vector is not in the open
G\aa rding cone $\Gamma_n$.
\end{proof}

\begin{remark}
For $r=n$ the speed is the Gauss--Kronecker curvature.
Proposition~\ref{prop:gauss-endpoint} is consistent with Urbas'
theorem that a complete noncompact convex contracting self-similar
solution of the $K^\alpha$-flow is a hyperplane
\cite[Theorem~2.1, p.~255]{Urbas1998}. Neither statement contains the
other: Proposition~\ref{prop:gauss-endpoint} assumes rotational symmetry
and graphicality of an end, but neither convexity nor completeness.
\end{remark}

\section{Construction and quantitative estimates}\label{sec:construction}

\subsection{Variation of parameters}
The construction uses the following integral representation.

\begin{lemma}\label{lem:variation}
Let $v$ solve
\begin{equation}\label{eq:linear}
v''+f(x)(v-xv')=g(x)
\end{equation}
on $(d,a)$.  Then
\begin{align}
v(x)={}&\frac{v(a)}a x
-\bigl(av'(a)-v(a)\bigr)x
 \int_x^a\frac1{s^2}
 \exp\left(-\int_s^a zf(z)\dd z\right)\dd s\notag\\
&+x\int_x^a\frac1{t^2}\int_t^a
sg(s)\exp\left(-\int_t^s zf(z)\dd z\right)\dd s\dd t.
\label{eq:variation}
\end{align}
\end{lemma}

\begin{proof}
Write $v=xy$.  Then
\[
xy''+(2-x^2f)y'=g.
\]
The equation for $w=y'$ is linear.  Multiplying it by the integrating factor
\[
x^2\exp\left(-\int^x zf(z)\dd z\right)
\]
and integrating first from $x$ to $a$ and then once more yields
\eqref{eq:variation}.
\end{proof}

For a fixed solution $u$ of \eqref{eq:profile}, equation \eqref{eq:linear}
holds with
\begin{equation}\label{eq:fg}
f=\kappa m^{r-1}\bigl(1+(u')^2\bigr),
\qquad
g=\frac{n-r}{r}\frac{1+(u')^2}{u}.
\end{equation}

\subsection{Finite-radius profiles and uniform estimates}

Fix $\sigma>0$ and set $a_0=(n-r)/r$. For every $a>0$, let $u_{\sigma,a}$ be the solution of
\eqref{eq:profile} with terminal data
\begin{equation*}
u_{\sigma,a}(a)=\sigma a,
\qquad
u_{\sigma,a}'(a)=\sigma.
\end{equation*}
The terminal term involving $av'(a)-v(a)$ in \eqref{eq:variation} vanishes.
For a positive function $v$, set
\[
\mathcal E_v(t,s)=
\exp\left(-\kappa\int_t^s
z\bigl(v(z)\sqrt{1+(v'(z))^2}\bigr)^{r-1}
\bigl(1+(v'(z))^2\bigr)\dd z\right).
\]
Hence,
\begin{align}
u_{\sigma,a}(x)=\sigma x+\frac{n-r}{r}x
\int_x^a\frac1{t^2}\int_t^a
\frac{s(1+(u_{\sigma,a}'(s))^2)}{u_{\sigma,a}(s)}
\mathcal E_{u_{\sigma,a}}(t,s)
\dd s\dd t.
\label{eq:finite-integral}
\end{align}

\begin{lemma}\label{lem:finite}
For every $a>0$, $u_{\sigma,a}$ extends to $(0,a]$.  Moreover,
\[
\Psi_{\sigma,a}<0\quad\text{on }(0,a)
\]
and, for $0<x\le a$,
\begin{align}
\sigma x\le u_{\sigma,a}(x)
&\le\sigma x+\frac{n-r}{2r\kappa\sigma^r}x^{-r},
\label{eq:finite0}\\
|u_{\sigma,a}'(x)-\sigma|
&\le\frac{3(n-r)}{2r\kappa\sigma^r}x^{-(r+1)}.
\label{eq:finite1}
\end{align}
\end{lemma}

\begin{proof}
At $x=a$, $\Psi_{\sigma,a}(a)=0$ and
\[
\Psi_{\sigma,a}'(a)=a u_{\sigma,a}''(a)>0.
\]
Thus $\Psi_{\sigma,a}<0$ immediately to the left of $a$.  If it had a largest
zero $x_1<a$, then $\Psi'(x_1)>0$, forcing $\Psi>0$ immediately to the right
of $x_1$, a contradiction.  Therefore $\Psi_{\sigma,a}<0$ on the maximal
left interval.  Hence $(u_{\sigma,a}/x)'<0$ and
$u_{\sigma,a}(x)\ge\sigma x$.

Fix $x$ in the maximal interval.  For $z\in[x,a]$,
\[
m_{\sigma,a}(z)\ge u_{\sigma,a}(z)\ge\sigma z\ge\sigma x.
\]
Set
\[
\theta'(z)=\kappa(\sigma x)^{r-1}z
\bigl(1+(u_{\sigma,a}'(z))^2\bigr).
\]
For $x\le t\le s\le a$ the integrand of the inner integral in
\eqref{eq:finite-integral} equals
\[
\frac{\theta'(s)}{\kappa(\sigma x)^{r-1}u_{\sigma,a}(s)}\,
\mathcal E_{u_{\sigma,a}}(t,s)
\le
\frac{\theta'(s)}{\kappa\sigma^r x^{r-1}t}\,
\exp\left(-\int_t^s\theta'(z)\dd z\right),
\]
because $u_{\sigma,a}(s)\ge\sigma s\ge\sigma t$ and, since
$m_{\sigma,a}(z)\ge\sigma x$, the exponent of
$\mathcal E_{u_{\sigma,a}}(t,s)$ is at most $-\int_t^s\theta'$.
Since
\[
\int_t^a\theta'(s)
\exp\left(-\int_t^s\theta'(z)\dd z\right)\dd s\le1,
\]
the inner integral is bounded by $1/(\kappa\sigma^r x^{r-1}t)$.  Integrating $t^{-3}$ gives
\eqref{eq:finite0}.  To record the derivative estimate and
its constant, write
\[
I_a(t):=\int_t^a
\frac{s(1+(u'_{\sigma,a}(s))^2)}{u_{\sigma,a}(s)}
\mathcal E_{u_{\sigma,a}}(t,s)\,\dd s.
\]
The preceding comparison, now made with base point $t$, gives
\begin{equation}\label{eq:Ia-bound}
0\le I_a(t)\le \frac1{\kappa\sigma^r t^r}
\qquad(0<t\le a).
\end{equation}
Formula \eqref{eq:finite-integral} becomes
$u_{\sigma,a}(x)-\sigma x=a_0x\int_x^a t^{-2}I_a(t)\,\dd t$.
Leibniz' rule therefore gives, including the lower-limit boundary term,
\[
u'_{\sigma,a}(x)-\sigma
=a_0\int_x^a t^{-2}I_a(t)\,\dd t-a_0\frac{I_a(x)}x.
\]
Using \eqref{eq:Ia-bound} separately in the integral and boundary terms,
\begin{align*}
|u'_{\sigma,a}(x)-\sigma|
&\le \frac{a_0}{\kappa\sigma^r}
\left(\int_x^\infty t^{-r-2}\,\dd t+x^{-r-1}\right)\\
&=\frac{a_0}{\kappa\sigma^r}
\left(1+\frac1{r+1}\right)x^{-r-1}
\le \frac{3a_0}{2\kappa\sigma^r}x^{-r-1},
\end{align*}
because $r\ge1$.  Since $a_0=(n-r)/r$, this is exactly
\eqref{eq:finite1}.

If the maximal left endpoint were $x_{\min}>0$, estimates
\eqref{eq:finite0}--\eqref{eq:finite1} would keep
$(u_{\sigma,a},u_{\sigma,a}')$ in a compact subset of
$(0,\infty)\times\R$.  The standard continuation theorem for ODEs would then
extend the solution past $x_{\min}$, a contradiction.
\end{proof}

\begin{remark}
Lemma~\ref{lem:finite} only asserts existence on $(0,a]$. Since the
right-hand side of \eqref{eq:profile} is smooth on $\{u>0\}$ for all
$x\in\R$, the terminal solutions, and the limit profile constructed
below, may continue to negative values of $x$; see
Proposition~\ref{prop:maximal-graph}. The sign argument for
$\Psi_{\sigma,a}$ uses $x>0$: at a zero of $\Psi$ one has
$\Psi'=xu''$, whose sign is that of $x$.
\end{remark}

\begin{proof}[Proof of Theorem~\ref{thm:main}: existence and estimates]
On every compact interval $[\varepsilon,E]\Subset(0,\infty)$, Lemma
\ref{lem:finite} gives uniform $C^1$ bounds for
$\{u_{\sigma,a}:a\ge E\}$ and a uniform positive lower bound for the
functions.  Since the right-hand side of \eqref{eq:profile} is smooth on
$(0,\infty)\times(0,\infty)\times\R$, the ODE gives uniform $C^k$ bounds on
each compact interval.  A diagonal Arzel\`a--Ascoli argument therefore
produces a sequence $a_j\to\infty$ for which
\[
u_{\sigma,a_j}\longrightarrow u_\sigma
\quad\text{in }C^\infty_{\mathrm{loc}}((0,\infty)).
\]
The limit solves \eqref{eq:profile} and inherits
\eqref{eq:main0}--\eqref{eq:main1}.

Passing to the limit in \eqref{eq:finite-integral} gives
\begin{align}
u_\sigma(x)=\sigma x+\frac{n-r}{r}x
\int_x^\infty\frac1{t^2}\int_t^\infty
\frac{s(1+(u_\sigma'(s))^2)}{u_\sigma(s)}
\mathcal E_{u_\sigma}(t,s)
\dd s\dd t.
\label{eq:global-integral}
\end{align}

We justify the passage to the limit by dominated convergence.  Fix $x>0$ and extend the
integrand for the terminal problem by zero when $s>a$.  Lemma~\ref{lem:finite}
gives, for every $a\ge x$ and $s\ge x$,
\[
u_{\sigma,a}(s)\ge\sigma s,
\qquad |u'_{\sigma,a}(s)|\le
P_x:=\sigma+\frac{3(n-r)}{2r\kappa\sigma^r}x^{-r-1}.
\]
For $x\le t\le s\le a$, the exponent in the kernel satisfies
\begin{align*}
\kappa\int_t^s z m_{\sigma,a}(z)^{r-1}
(1+(u'_{\sigma,a}(z))^2)\,\dd z
&\ge \kappa\sigma^{r-1}\int_t^s z^r\,\dd z\\
&\ge C_x(s^2-t^2),
\end{align*}
where
\[
C_x:=\frac{\kappa\sigma^{r-1}x^{r-1}}2>0.
\]
Also
\[
0\le \frac{s(1+(u'_{\sigma,a}(s))^2)}{u_{\sigma,a}(s)}
\le D_x:=\frac{1+P_x^2}{\sigma}.
\]
Thus the absolute value of the double-integral density is bounded,
independently of $a$, by
\[
D_x t^{-2}e^{-C_x(s^2-t^2)}\mathbf 1_{\{x\le t\le s\}}.
\]
This function is integrable: since $s^2-t^2\ge t(s-t)$,
\[
\int_x^\infty t^{-2}\int_t^\infty
e^{-C_x(s^2-t^2)}\,\dd s\,\dd t
\le \frac1{C_x}\int_x^\infty t^{-3}\,\dd t<\infty.
\]
Local smooth convergence supplies pointwise convergence of the integrands.
The dominated convergence theorem on the fixed region
$\{(t,s):x\le t\le s<\infty\}$ therefore proves
\eqref{eq:global-integral}; in particular, the majorant does not depend on
the terminal radius.

Define
\[
J(t)=\int_t^\infty
\frac{s(1+(u_\sigma'(s))^2)}{u_\sigma(s)}
\exp\left(-\kappa\int_t^s
z m_\sigma(z)^{r-1}(1+(u_\sigma'(z))^2)\dd z\right)\dd s.
\]
Differentiating \eqref{eq:global-integral} gives the exact identity
\begin{equation}\label{eq:PsiJ}
\Psi_\sigma(x)=xu_\sigma'(x)-u_\sigma(x)
=-\frac{n-r}{r}J(x)<0.
\end{equation}
The comparison used in Lemma~\ref{lem:finite}, now with base point $x$, yields
\[
J(x)\le\frac1{\kappa\sigma^r x^r},
\]
which proves \eqref{eq:mainPsi}.  Proposition~\ref{prop:admissibility} now
shows that the curvature vector belongs to $\Ga_r$ at every point.  Thus the
elliptic solution generates a homothetically shrinking solution of the
$S_r$-flow.
\end{proof}

\section{Uniqueness for a prescribed conical slope}\label{sec:uniqueness}

Throughout this section, let $1\le r\le n-1$ and set $a_0=(n-r)/r>0$.
The argument follows the integral approach to uniqueness of
Kleene--M\o ller \cite[Section~2, Proposition~1, p.~3951]{KleeneMoller2014},
adapted to the additional dependence of the coefficients on the profile
when $r>1$.
We first show that the convergence of $u(x)/x$ to a positive slope already
forces the decay needed for uniqueness.  Constants in the following lemma
may depend on the solution; the explicit constants in
Theorem~\ref{thm:main} are not needed in this step.

\begin{lemma}\label{lem:automatic-decay}
Let $d\in\R$ and let $u\in C^2((d,\infty))$ be a positive solution of
\eqref{eq:profile} such that $u(x)/x\to\sigma>0$, and put
$d_+=\max\{d,0\}$.  Then
$\Psi=xu'-u<0$ and $u(x)>\sigma x$ on $(d_+,\infty)$.  Moreover,
$u'>0$ for all sufficiently large $x$, and
\begin{equation}\label{eq:automatic-decay}
u-\sigma x=O(x^{-r}),\qquad
u'-\sigma=O(x^{-r-1}),\qquad
\Psi=O(x^{-r})
\end{equation}
as $x\to\infty$.
\end{lemma}

\begin{proof}
Write $p=u'$ and define
\[
\mathcal A(x)=\kappa x u^{r-1}(1+p^2)^{(r+1)/2},\qquad
b(x)=a_0\frac{x(1+p^2)}u.
\]
Since $\Psi'=xu''$, the profile equation gives
\begin{equation}\label{eq:defect-linear-unique}
\Psi'=\mathcal A\Psi+b.
\end{equation}
Here $b>0$ on $(d_+,\infty)$, and $u(x)\ge(\sigma/2)x$ for all
sufficiently large $x$ implies
\[
\mathcal A(x)\ge c_0x^r,\qquad c_0=\kappa(\sigma/2)^{r-1}>0.
\]
We first prove negativity, using only the conical asymptote.
If $\Psi(x_0)\ge0$ at some point $x_0>d_+$, the integrating-factor formula for
\eqref{eq:defect-linear-unique} implies $\Psi(x)>0$ for $x>x_0$.
Choosing $x_1>x_0$ sufficiently large then gives
\[
\Psi(x)\ge\Psi(x_1)
\exp\left(\frac{c_0}{r+1}(x^{r+1}-x_1^{r+1})\right),
\qquad x\ge x_1.
\]
Since $(u/x)'=\Psi/x^2$, this contradicts the finite limit of $u/x$.
Thus $\Psi<0$ on $(d_+,\infty)$, and $u/x$ decreases strictly to
$\sigma$ there.  In particular, $u>\sigma x$ and $p<u/x$ on $(d_+,\infty)$.

At any zero of $p$, the profile equation reads
\[
p'=\frac{a_0}{u}-\kappa u^r.
\]
Because $u(x)\to\infty$, choose $R>d_+$ so that every zero of $p$ in
$[R,\infty)$ satisfies $p'<0$.  Thus $p$ cannot cross from negative
to positive values there.  More explicitly, if
$p(x_0)<0$ and $p$ later became nonnegative, its first subsequent zero $z$
would satisfy $p(z)=0$ and $p'(z)\ge0$, contrary to the choice of $R$.
If $p(x_0)=0$, then $p'(x_0)<0$, so $p<0$ immediately to the right and the
same first-zero argument applies.  Hence $p(x)\le0$ for every $x\ge x_0$.
It follows that $u$ is nonincreasing on $[x_0,\infty)$, contradicting
$u(x)/x\to\sigma>0$, which implies $u(x)\to\infty$.  Therefore $p$ has no
nonpositive value on a sufficiently distant half-line.
Hence, after enlarging $R$,
\[
0<p(x)<\frac{u(x)}x\le M,\qquad x\ge R,
\]
for some $M>0$.  It follows that $|\Psi(x)|=O(x)$ and
$0<b(x)\le B_0$ on this half-line.

Integrating \eqref{eq:defect-linear-unique} from $x$ to $T$ gives
\[
\Psi(x)=e^{-\int_x^T \mathcal A(t)\dd t}\Psi(T)
-\int_x^T b(s)e^{-\int_x^s \mathcal A(t)\dd t}\dd s.
\]
The first term tends to zero because $\Psi(T)=O(T)$ and
$\mathcal A(t)\ge c_0t^r$.  The integral converges by the same lower bound and
the boundedness of $b$.  Hence
\begin{equation}\label{eq:defect-infinity-unique}
\Psi(x)=-\int_x^\infty b(s)
\exp\left(-\int_x^s \mathcal A(t)\dd t\right)\dd s.
\end{equation}
For $s\ge x\ge R$ we have $\mathcal A(s)\ge c_0x^r$, and therefore
\[
0<-\Psi(x)\le B_0\int_x^\infty e^{-c_0x^r(s-x)}\dd s
=\frac{B_0}{c_0}x^{-r}.
\]
Finally,
\[
\frac{u(x)}x-\sigma=-\int_x^\infty\frac{\Psi(t)}{t^2}\dd t,
\qquad
u'-\sigma=\frac{u-\sigma x+\Psi}{x}.
\]
These identities give
$0<u-\sigma x\le B_0[c_0(r+1)]^{-1}x^{-r}$ and the asserted
estimate for $u'-\sigma$.
\end{proof}

\begin{theorem}\label{thm:uniqueness}
Let $1\le r\le n-1$, $\kappa>0$, and $\sigma>0$.  Suppose that
$u_i\in C^2((d_i,\infty))$, $d_i\in\R$, $i=1,2$, are positive
solutions of \eqref{eq:profile} and
\[
\lim_{x\to\infty}\frac{u_1(x)}x
=\lim_{x\to\infty}\frac{u_2(x)}x=\sigma.
\]
Then $u_1=u_2$ on $(d_3,\infty)$, \(d_3=\max\{d_1,d_2\}\).
\end{theorem}

\begin{proof}
Set
\[
F(x,u,p)=\frac{a_0(1+p^2)}u
+\kappa u^{r-1}(1+p^2)^{(r+1)/2}(xp-u).
\]
Thus $u_i''=F(x,u_i,u_i')$.  For $0\le\theta\le1$, define
\[
u_\theta=(1-\theta)u_2+\theta u_1,\qquad p_\theta=u_\theta',
\]
and let
\[
\alpha(x)=\int_0^1F_u(x,u_\theta,p_\theta)\dd\theta,\qquad
\beta(x)=\int_0^1F_p(x,u_\theta,p_\theta)\dd\theta.
\]
The fundamental theorem of calculus implies that $w=u_1-u_2$ satisfies
$w''=\alpha w+\beta w'$.  The interpolation $u_\theta$ need not solve
the profile equation.  Nevertheless, Lemma~\ref{lem:automatic-decay}
implies, uniformly for $\theta\in[0,1]$,
\begin{equation}\label{eq:interpolated-decay}
u_\theta=\sigma x+O(x^{-r}),\qquad
p_\theta=\sigma+O(x^{-r-1}),\qquad
xp_\theta-u_\theta=O(x^{-r}).
\end{equation}

The appropriate unknown is $y=w/x$.  Direct substitution gives
\begin{equation}\label{eq:normalized-difference}
y''-P(x)y'=Q(x)y,\qquad
P=\beta-\frac2x,\qquad Q=\alpha+\frac\beta x.
\end{equation}
We verify the estimates
\begin{equation}\label{eq:uniqueness-coefficients}
P(x)\ge c_*x^r,\qquad |Q(x)|\le C_*x^{-2}
\end{equation}
on a sufficiently distant half-line, with $c_*,C_*>0$.
For this calculation write $W=(1+p^2)^{1/2}$ and $\Psi=xp-u$.  Then
\begin{align*}
F_u&=-\frac{a_0W^2}{u^2}
 +\kappa(r-1)u^{r-2}W^{r+1}\Psi
 -\kappa u^{r-1}W^{r+1},\\
F_p&=\frac{2a_0p}{u}+\kappa xu^{r-1}W^{r+1}
 +\kappa(r+1)u^{r-1}pW^{r-1}\Psi.
\end{align*}
Using \eqref{eq:interpolated-decay}, we obtain
\[
\beta(x)=\kappa\sigma^{r-1}(1+\sigma^2)^{(r+1)/2}x^r+O(x^{-1}),
\]
which yields the lower bound for $P$.  More importantly, the leading
terms in $F_u+F_p/x$ cancel exactly:
\begin{align}
F_u+\frac{F_p}{x}
={}&a_0\left(\frac{2p}{xu}-\frac{W^2}{u^2}\right)+\kappa\Psi\left((r-1)u^{r-2}W^{r+1}
 +\frac{r+1}{x}u^{r-1}pW^{r-1}\right).
\label{eq:uniqueness-cancellation}
\end{align}
At $(u,p)=(u_\theta,p_\theta)$ the first term on the right-hand side
is $O(x^{-2})$; in the second term the factor in parentheses is
$O(x^{r-2})$ and
$\Psi=O(x^{-r})$.  Integration in $\theta$ gives the bound for $Q$.
This calculation also applies to $r=1$, when the term with factor $r-1$
vanishes.

Lemma~\ref{lem:automatic-decay} gives
\[
y=O(x^{-r-1}),\qquad y'=O(x^{-r-2}).
\]
Integrating \eqref{eq:normalized-difference} from $x$ to $T$ and
letting $T\to\infty$, the terminal term
$y'(T)\exp(-\int_x^T P)$ tends to zero by
\eqref{eq:uniqueness-coefficients}.  We obtain
\begin{equation}\label{eq:difference-infinity}
y'(x)=-\int_x^\infty Q(s)y(s)
\exp\left(-\int_x^sP(t)\dd t\right)\dd s.
\end{equation}
The integral is absolutely convergent because $y$ is bounded and
\eqref{eq:uniqueness-coefficients} holds.

Choose $R$ so that \eqref{eq:uniqueness-coefficients} holds for $x\ge R$
and set $M_R=\sup_{s\ge R}|y(s)|<\infty$.  For $x\ge R$,
\begin{align*}
|y'(x)|
&\le C_*M_R\int_x^\infty s^{-2}e^{-c_*x^r(s-x)}\dd s\\
&\le\frac{C_*}{c_*}M_Rx^{-r-2}.
\end{align*}
Since $y(x)\to0$, a second integration gives
\[
|y(x)|\le\frac{C_*}{c_*(r+1)}M_Rx^{-r-1}.
\]
Taking the supremum over $x\ge R$, we conclude that
\begin{equation}\label{eq:tail-contraction}
M_R\le\frac{C_*}{c_*(r+1)}R^{-r-1}M_R.
\end{equation}
Enlarging $R$ so that the coefficient is less than one yields $M_R=0$.
Thus $u_1=u_2$ on $[R,\infty)$.  Since $F$ is smooth for $u>0$,
uniqueness for the initial value problem extends this equality throughout
the common interval of definition.
\end{proof}

Theorem~\ref{thm:uniqueness} is the uniqueness assertion of
Theorem~\ref{thm:main}: a positive solution on $(d,\infty)$ with
$u(x)/x\to\sigma$ agrees with $u_\sigma$ on the common interval of
definition.

\begin{corollary}\label{cor:canonical-limit}
For each $\sigma>0$, the finite-radius profiles converge without passing
to a subsequence:
\[
u_{\sigma,a}\longrightarrow u_\sigma
\quad\text{in }C^\infty_{\mathrm{loc}}((0,\infty))
\quad\text{as }a\to\infty.
\]
In particular, the limit does not depend on the choice of the sequence $a_j$.
\end{corollary}

\begin{proof}
The bounds in Lemma~\ref{lem:finite} and the profile equation imply
precompactness on every compact subinterval, as in the existence proof.
Every sequence $a_j\to\infty$ has a subsequence converging locally
smoothly to a positive solution with asymptotic slope $\sigma$.
Theorem~\ref{thm:uniqueness} identifies each such limit with $u_\sigma$.
If convergence of the full family failed on some compact interval in
some $C^k$ norm, a sequence witnessing that failure would have a
subsequence converging to $u_\sigma$, a contradiction.
\end{proof}

\begin{proposition}\label{prop:slope-dependence}
If $0<\sigma<\sigma'$, then
\begin{equation}\label{eq:explicit-ordering-radius}
u_\sigma(x)<u_{\sigma'}(x)
\quad\text{for every}\quad
x\ge R(\sigma,\sigma'):=
\left(
\frac{n-r}{2r\kappa\sigma^r(\sigma'-\sigma)}
\right)^{1/(r+1)}.
\end{equation}
In particular, the family is strictly ordered outside an explicitly
controlled compact region for every fixed pair of slopes.
\end{proposition}

\begin{proof}
By Theorem~\ref{thm:main},
\[
u_{\sigma'}(x)>\sigma'x,
\qquad
u_\sigma(x)\le\sigma x+
\frac{n-r}{2r\kappa\sigma^r}x^{-r}.
\]
Consequently,
\[
u_{\sigma'}(x)-u_\sigma(x)
>(\sigma'-\sigma)x-
\frac{n-r}{2r\kappa\sigma^r}x^{-r},
\]
and the right-hand side is nonnegative exactly when $x\ge R(\sigma,\sigma')$.
\end{proof}

Continuity, and in fact real analyticity, of $\sigma\mapsto u_\sigma$ is
proved in Section~\ref{sec:analytic}. The ordering in
Proposition~\ref{prop:slope-dependence} does not extend to all $x>0$; see
Remark~\ref{rem:numerical-crossing}.

\section{Complete asymptotic expansion}\label{sec:expansion}

Let $u$ be a positive solution on $(d,\infty)$ with $u(x)/x\to\sigma>0$.
Throughout this section set $a_0=(n-r)/r$, $q=r+1$, and
$W_0=\sqrt{1+\sigma^2}$.  The symbol $q$ is distinct from the quantity
$m=u\sqrt{1+(u')^2}$ used in the profile equation.
All estimates below are taken as $x\to\infty$.

\begin{proposition}\label{prop:defect-leading}
The angular defect satisfies
\begin{equation}\label{eq:formalPsi}
\Psi(x)\sim-\frac{n-r}
{r\kappa\sigma^r(1+\sigma^2)^{(r-1)/2}}x^{-r}.
\end{equation}
Thus, $u''=o(x^{-1})$ and $\mu=o(x^{-1})$.
\end{proposition}

\begin{proof}
Use $\mathcal A$ and $b$ from the proof of Lemma~\ref{lem:automatic-decay}.
That lemma gives
\[
\frac{\mathcal A(x)}{x^r}\longrightarrow \mathcal A_\infty
=\kappa\sigma^{r-1}W_0^{r+1},\qquad
b(x)\longrightarrow b_\infty=\frac{a_0W_0^2}{\sigma}.
\]
In \eqref{eq:defect-infinity-unique}, set $s=x+t/x^r$.  Then
\[
x^r\Psi(x)=-\int_0^\infty b(x+t/x^r)
\exp\left(-\int_x^{x+t/x^r}\mathcal A(z)\dd z\right)\dd t.
\]
For every fixed $t\ge0$, the exponent tends to $-\mathcal A_\infty t$.
Indeed, after setting $z=x+\tau/x^r$, its negative is
$\int_0^t x^{-r}\mathcal A(x+\tau/x^r)\dd\tau$, whose integrand converges
uniformly on bounded $\tau$-intervals to $\mathcal A_\infty$.
The bounds $\mathcal A(z)\ge \underline c\, z^r$ and $|b(z)|\le C$ for large $z$ dominate
the integrand by $Ce^{-\underline c\,t}$.  Dominated convergence therefore yields
\[
\lim_{x\to\infty}x^r\Psi(x)
=-\int_0^\infty b_\infty e^{-\mathcal A_\infty t}\dd t
=-\frac{b_\infty}{\mathcal A_\infty},
\]
which proves \eqref{eq:formalPsi}.  It follows that
\[
\frac{a_0}{u}+\kappa
\bigl(u\sqrt{1+(u')^2}\bigr)^{r-1}\Psi=o(x^{-1}).
\]
Since $u'\to\sigma$, the profile equation and
\eqref{eq:curv-graph} give the remaining conclusions.  This proof does
not require differentiating an asymptotic remainder.
\end{proof}

We next prove that the expansion extends to all orders, with derivative
bounds on the remainders.  The argument also proves
Theorem~\ref{thm:expansion-intro} for any positive solution with limiting
slope $\sigma$.

\begin{proof}[Proof of Theorem~\ref{thm:expansion-intro}]
\emph{Construction of the coefficients.}
Put $t=x^{-q}$ and write $u(x)=xv(t)$.  Dots denote derivatives with
respect to $t$.  Since $t'=-qx^{-q-1}=-qt/x$, the chain rule
first gives
$u'=v+x\dot v\,t'=v-qt\dot v$.  Hence
$\Psi=xu'-u=-qxt\dot v$, and a second differentiation gives
$u''=(qt/x)(r\dot v+qt\ddot v)$ because $q-1=r$.  Equivalently,
\[
u'=v-qt\dot v,\qquad
\Psi=-qxt\dot v,\qquad
xu''=qt(r\dot v+qt\ddot v).
\]
Writing $W=(1+(v-qt\dot v)^2)^{1/2}$, the profile equation becomes
\begin{equation}\label{eq:transformed-expansion}
\mathcal H[v]:=
\kappa qv^{r-1}W^{r+1}\dot v
+qt(r\dot v+qt\ddot v)-\frac{a_0W^2}{v}=0.
\end{equation}
The expression is analytic near $t=0$, $v=\sigma>0$, for bounded
$\dot v$ and $\ddot v$.  Seek a formal series
$v=\sigma+\sum_{j\ge1}c_jt^j$.  Its constant coefficient in
\eqref{eq:transformed-expansion} gives
\[
\kappa q\sigma^{r-1}W_0^{r+1}c_1
=\frac{a_0W_0^2}{\sigma},
\]
which is \eqref{eq:first-coefficient}.

The second coefficient is
\begin{equation}\label{eq:second-coefficient}
c_2=-\frac{r c_1^2}{2}
\left[
\frac1\sigma+
\frac{\sigma}{1+\sigma^2}
\left(1-r+\frac{r(r+1)}{n-r}\right)
\right].
\end{equation}
To verify this formula, take the coefficient of $t$.  Indeed,
put $S=1+\sigma^2=W_0^2$.  From
\[
v=\sigma+c_1t+c_2t^2+O(t^3),\qquad
v-qt\dot v=\sigma-r c_1t+O(t^2),
\]
we obtain
\begin{align*}
v^{r-1}W^{r+1}
&=\sigma^{r-1}S^{(r+1)/2}
\left[1+c_1\left(\frac{r-1}{\sigma}
-\frac{r(r+1)\sigma}{S}\right)t+O(t^2)\right],\\
\frac{W^2}{v}
&=\frac{S}{\sigma}
+c_1\left(-2r-\frac{S}{\sigma^2}\right)t+O(t^2).
\end{align*}
Taking the coefficient of $t$ in \eqref{eq:transformed-expansion} yields
\begin{align*}
0={}&\kappa q\sigma^{r-1}S^{(r+1)/2}
\left[2c_2+c_1^2\left(\frac{r-1}{\sigma}
-\frac{r(r+1)\sigma}{S}\right)\right]\\
&+qr c_1+a_0c_1\left(2r+\frac{S}{\sigma^2}\right).
\end{align*}
Using
$\kappa q\sigma^{r-1}S^{(r+1)/2}c_1=a_0S/\sigma$
and $a_0=(n-r)/r$, elementary simplification gives exactly
\eqref{eq:second-coefficient}.

For $j\ge1$, the coefficient of $t^{j-1}$ involving $c_j$ is exactly
$\kappa qj\sigma^{r-1}W_0^{r+1}c_j$.  To see this, the new term in
$v$ is of order $t^j$, its contribution to $v-qt\dot v$ is also of
order $t^j$, and its contribution to
$qt(r\dot v+qt\ddot v)$ is of order $t^j$.  Only the factor
$\dot v$ in the first term of \eqref{eq:transformed-expansion}
contributes at order $t^{j-1}$.  Thus, with $v_{j-1}(t)=
\sigma+\sum_{i=1}^{j-1}c_it^i$, the precise recursion is
\begin{equation}\label{eq:coefficient-recursion}
c_j=-\frac{[t^{j-1}]\mathcal H[v_{j-1}]}
{\kappa qj\sigma^{r-1}W_0^{r+1}}.
\end{equation}
Here $[t^{j-1}]$ denotes the Taylor coefficient at zero.  The denominator
is positive, so the recursion uniquely determines every coefficient.
No logarithmic terms are needed.

For $N\ge1$ set
\[
U_N(x)=\sigma x+\sum_{j=1}^N c_jx^{1-jq}.
\]
If $F$ is the right-hand side of the profile equation, as defined in the
proof of Theorem~\ref{thm:uniqueness}, then
\begin{equation}\label{eq:truncation-residual}
E_N(x):=U_N''-F(x,U_N,U_N')=O(x^{-1-Nq}).
\end{equation}
Indeed, $xE_N=\mathcal H[v_N](x^{-q})$, and the recursion cancels
all Taylor coefficients of $\mathcal H[v_N]$ through degree $N-1$.
This constructs formal approximations only; the following estimate compares them
with the actual solution.

\emph{Control of the remainder and its first derivative.}
For large $x$, $U_N>0$.  Define $e_N=u-U_N$ and $y_N=e_N/x$.
Linearizing $F$ along the segment joining $(U_N,U_N')$ to $(u,u')$
as in Theorem~\ref{thm:uniqueness} gives
\begin{equation}\label{eq:remainder-equation}
y_N''-P_Ny_N'=Q_Ny_N-\frac{E_N}{x}.
\end{equation}
Both endpoints of this segment satisfy the estimates in
\eqref{eq:interpolated-decay}.  Therefore the algebraic identity
\eqref{eq:uniqueness-cancellation}, which does not require the
interpolating profiles to solve the equation, yields
\begin{equation}\label{eq:remainder-coefficients}
P_N(x)\ge \underline c\, x^r,\qquad |Q_N(x)|\le Cx^{-2}
\end{equation}
for all sufficiently large $x$.  Constants may depend on $N$.
By Lemma~\ref{lem:automatic-decay}, initially
$y_N=O(x^{-q})$ and $y_N'=O(x^{-q-1})$.  Thus integration from infinity
in \eqref{eq:remainder-equation} is justified and gives
\begin{equation}\label{eq:remainder-integral}
y_N'(x)=-\int_x^\infty
\left(Q_N(s)y_N(s)-\frac{E_N(s)}s\right)
\exp\left(-\int_x^sP_N(z)\dd z\right)\dd s.
\end{equation}
The exponential lower bound in \eqref{eq:remainder-coefficients}
kills the terminal term and makes the integral absolutely convergent.

Let $M_R=\sup_{s\ge R}|y_N(s)|$.  By
\eqref{eq:truncation-residual}--\eqref{eq:remainder-integral}, for $x\ge R$,
\begin{align*}
|y_N'(x)|
&\le C\int_x^\infty
\left(M_Rs^{-2}+s^{-2-Nq}\right)e^{-\underline c\,x^r(s-x)}\dd s\\
&\le C M_Rx^{-r-2}+C x^{-r-2-Nq}.
\end{align*}
Integrating once more, using $y_N(\infty)=0$ and $q=r+1$, gives
\[
|y_N(x)|\le C M_Rx^{-q}+C x^{-(N+1)q}.
\]
The constants here are independent of $R$ once it is sufficiently large.
Taking the supremum for $x\ge R$ and then increasing $R$ if necessary,
\[
M_R\le C R^{-q}M_R+C R^{-(N+1)q}
\quad\Longrightarrow\quad
M_R\le C'R^{-(N+1)q}.
\]
Since this holds for every sufficiently large $R$, it proves
$y_N=O(x^{-(N+1)q})$.  Substituting this pointwise estimate back into
\eqref{eq:remainder-integral} gives
$y_N'=O(x^{-1-(N+1)q})$.  Hence
\begin{equation}\label{eq:remainder-C1}
u-U_N=O(x^{1-(N+1)q}),\qquad
u'-U_N'=O(x^{-(N+1)q}).
\end{equation}

\emph{Control after arbitrary differentiation.}
We supply this step separately, since differentiating a big-$O$ estimate
without further information is not valid.  The smoothness of $F$ for
$u>0$ and the bounds $u\asymp x$, $u'=O(1)$ first give
$u''=O(x^{-1})$.  Repeated differentiation of the ODE then shows that,
for each integer $\ell\ge0$, some finite exponent $B_\ell$ satisfies
\[
|u^{(\ell)}(x)|=O(x^{B_\ell}).
\]
For completeness, every partial derivative of $F(x,u,p)$ has at most
polynomial growth when $u\asymp x$ and $p$ is bounded: the denominators
are powers of $u$ and $1+p^2$, and the remaining factors are powers of
$x$, $u$, and bounded functions of $p$.  The chain rule therefore gives
the assertion inductively from the estimates for $u$ and $u'$.

For fixed $k$, the $2k$th derivative of the linear term $\sigma x$ vanishes,
and every remaining monomial in $U_M^{(2k)}$ has exponent
$1-jq-2k\le1-q-2k<0$. Hence,
$|U_M^{(2k)}(x)|\le C_{M,k}$ for $x\ge1$: the constant may depend on $M$,
but the polynomial exponent does not.  Taking
$B=\max\{B_{2k},0\}$ therefore gives, independently of $M$,

\[
\|(u-U_M)^{(2k)}\|_{L^\infty([x-1,x+1])}\le C_M x^B.
\]

We record the interpolation estimate used here.  If $f\in C^{2k}(I)$ and
$[x-1,x+1]\subset I$, then
\[
|f^{(k)}(x)|\le C_k\left(
\|f\|_{L^\infty(I_x)}+
\|f\|_{L^\infty(I_x)}^{1/2}
\|f^{(2k)}\|_{L^\infty(I_x)}^{1/2}\right),
\qquad I_x=[x-1,x+1].
\]
Indeed, Taylor interpolation at $2k$ fixed distinct nodes, rescaled into
$[x-h,x+h]$ with $0<h\le1$, yields
\[
|f^{(k)}(x)|\le C_k\bigl(h^{-k}\|f\|_\infty
+h^k\|f^{(2k)}\|_\infty\bigr).
\]
If $\|f^{(2k)}\|_\infty\le\|f\|_\infty$, take $h=1$.  Otherwise take
$h=(\|f\|_\infty/\|f^{(2k)}\|_\infty)^{1/(2k)}$.  These two cases give the
displayed inequality.  This is the local one-dimensional
Landau--Kolmogorov interpolation inequality; the argument above is included
so that no differentiation of a big-$O$ remainder is being assumed.

Applying it to $f=u-U_M$, using \eqref{eq:remainder-C1}, gives
\[
|(u-U_M)^{(k)}(x)|
\le C_{M,k}\left(x^{1-(M+1)q}
+x^{(1-(M+1)q+B)/2}\right).
\]
Given $N$ and $k$, choose $M>N$ large enough that both exponents are at
most $1-(N+1)q-k$.  Since
\[
u-U_N=(u-U_M)+\sum_{j=N+1}^M c_jx^{1-jq},
\]
we obtain
\[
(u-U_N)^{(k)}=O(x^{1-(N+1)q-k}).
\]
Together with \eqref{eq:remainder-C1}, this proves
\eqref{eq:expansion-intro} for every $N$ and $k$.

Finally, taking $N=1$ and $k=0,1,2$ yields
\begin{align*}
u&=\sigma x+c_1x^{-r}+O(x^{-2r-1}),\\
u'&=\sigma-rc_1x^{-r-1}+O(x^{-2r-2}),\\
u''&=r(r+1)c_1x^{-r-2}+O(x^{-2r-3}).
\end{align*}
Using $\mu=-u''(1+(u')^2)^{-3/2}$ gives
\eqref{eq:meridian-expansion} and completes the proof.
\end{proof}

\begin{corollary}\label{cor:sharp-defect}
The angular defect has the sharper expansion
\begin{equation}\label{eq:sharp-defect}
\Psi(x)=-\frac{n-r}{r\kappa\sigma^r(1+\sigma^2)^{(r-1)/2}}
 x^{-r}+O(x^{-2r-1}).
\end{equation}
In particular, $u''>0$ and $\mu<0$ for all sufficiently large $x$.
\end{corollary}

\begin{proof}
Use $\Psi=xu'-u$, the first two expansions above, and $c_1>0$.
The signs follow from the expansion of $u''$ and
\eqref{eq:meridian-expansion}.
\end{proof}

\begin{remark}
The preceding result gives a complete classical asymptotic
expansion, differentiable term by term, with exponent set
$\{1-j(r+1):j\ge0\}$ and without logarithmic terms.  Here ``complete'' means
precisely the remainder estimate \eqref{eq:expansion-intro} for every finite
truncation and every derivative order; it does not assert convergence of the
infinite formal series.  The proof of the leading defect in
Proposition~\ref{prop:defect-leading} is independent of the construction
of the higher-order coefficients.
\end{remark}

\section{Analytic dependence and exterior foliation}\label{sec:analytic}

The first-order drift in the profile equation grows like $x^r$.
Consequently, its inverse must not be treated as a uniformly bounded
two-derivative gain in ordinary dyadically weighted H\"older spaces.
We first use a graph norm, which requires only the weighted $C^1$
estimates supplied by the integral equation. The expansion proved in
Section~\ref{sec:expansion} then upgrades parameter analyticity to all
weighted derivative orders.

For $R>0$, let $\mathcal C^j_\delta(R)$ be the space of functions on
$[R,\infty)$ with norm
\[
 \|z\|_{\mathcal C^j_\delta(R)}
 =\sum_{i=0}^j\sup_{x\ge R}x^{\delta+i}|z^{(i)}(x)|.
\]
For $0<\alpha<1$, define
\[
 \|z\|_{\mathcal X^{k,\alpha}_\delta(R)}
 =\|z\|_{\mathcal C^k_\delta(R)}
 +\sup_{x\ge R}x^{\delta+k+\alpha}
 [z^{(k)}]_{C^{0,\alpha}([x,2x])}.
\]
All spaces below also have their complex-valued versions.

\subsection{The linear inverse in a graph norm}

Fix $\sigma_0>0$ and write
\[
 \alpha_0(x)=F_u(x,u_{\sigma_0},u_{\sigma_0}'),\qquad
 \beta_0(x)=F_p(x,u_{\sigma_0},u_{\sigma_0}'),\qquad
 L_0z=z''-\alpha_0z-\beta_0z'.
\]
Here $F$ is the right-hand side of \eqref{eq:profile}, as defined in
the proof of Theorem~\ref{thm:uniqueness}. Let
\[
 \mathcal D_0(R)
 =\{z\in C^2([R,\infty)):
       z\in\mathcal C^1_r(R),\ L_0z\in\mathcal C^0_1(R)\},
\]
with the graph norm
\[
 \|z\|_{\mathcal D_0(R)}
 =\|z\|_{\mathcal C^1_r(R)}+\|L_0z\|_{\mathcal C^0_1(R)}.
\]
This is a Banach space: a Cauchy sequence converges in weighted $C^1$,
and the identity $z''=L_0z+\alpha_0z+\beta_0z'$ gives convergence of
its second derivatives on every compact interval. No boundary condition
is imposed at $R$; the two conditions at infinity are contained in the
weighted $C^1$ norm.

\begin{lemma}\label{lem:slope-linear-inverse}
For sufficiently large $R$, the operator
\[
 L_0:\mathcal D_0(R)\longrightarrow\mathcal C^0_1(R)
\]
is an isomorphism. Its weighted $C^1$ estimate can be chosen uniformly
when $\sigma_0$ ranges over a compact subinterval of $(0,\infty)$.
\end{lemma}

\begin{proof}
Put $z=xy$. The equation $L_0z=h$ becomes
\begin{equation}\label{eq:forced-normalized-linearization}
 y''-P_0y'=Q_0y+\frac hx,\qquad
 P_0=\beta_0-\frac2x,\quad Q_0=\alpha_0+\frac{\beta_0}x.
\end{equation}
The calculation in \eqref{eq:uniqueness-cancellation} gives, for large $R$,
\begin{equation}\label{eq:uniform-linear-coefficients}
 P_0(x)\ge \underline c\,x^r,\qquad |Q_0(x)|\le Cx^{-2}.
\end{equation}
These constants are uniform on compact slope intervals, by
\eqref{eq:main0}--\eqref{eq:mainPsi}.
A decaying solution must satisfy
\begin{align}
 y'(x)&=-\int_x^\infty
 \left(Q_0(s)y(s)+\frac{h(s)}s\right)
 e^{-\int_x^sP_0(t)\,dt}\,ds,
 \label{eq:linear-volterra-derivative}\\
 y(x)&=-\int_x^\infty y'(t)\,dt.
 \label{eq:linear-volterra-function}
\end{align}
The terminal term vanishes for $z\in\mathcal D_0(R)$.
Conversely, a fixed point of these integrals is $C^2$ and solves
\eqref{eq:forced-normalized-linearization}.

Use $\|y\|=\sup_{x\ge R}x^{r+1}|y(x)|$ and
$\|h\|_1=\sup_{x\ge R}x|h(x)|$.
The estimate $\int_x^sP_0\ge \underline c\,x^r(s-x)$ gives
\[
 |y'(x)|\le C\|h\|_1x^{-r-2}+C\|y\|x^{-2r-3},\qquad
 |y(x)|\le C\|h\|_1x^{-r-1}+C\|y\|x^{-2r-2}.
\]
The part depending on $y$ therefore has operator norm at most
$CR^{-r-1}$. Choose $R$ so this is less than $1/2$ and apply the
contraction principle in weighted $C^0$. The resulting solution satisfies
\begin{equation}\label{eq:linear-weighted-C0}
 \|z\|_{\mathcal C^1_r(R)}\le C\|h\|_1.
\end{equation}
Since $L_0z=h$, it belongs to $\mathcal D_0(R)$ and its graph norm is
bounded by $(C+1)\|h\|_1$. Taking $h=0$ proves injectivity.
\end{proof}

\subsection{Analyticity in weighted \texorpdfstring{$C^1$}{C1}}

\begin{lemma}\label{lem:analytic-profile-operator}
For $R$ sufficiently large, the map
\[
 \mathscr G(\sigma,v)
 =v''-F(x,\sigma x+v,\sigma+v')
\]
is real analytic from a neighborhood of
$(\sigma_0,u_{\sigma_0}-\sigma_0x)$ in
$\mathbb R\times\mathcal D_0(R)$ to $\mathcal C^0_1(R)$.
It has a holomorphic extension to a complex neighborhood.
\end{lemma}

\begin{proof}
Set $v_0=u_{\sigma_0}-\sigma_0x$. By
Theorem~\ref{thm:expansion-intro}, $v_0\in\mathcal C^1_r(R)$ and
$v_0''=O(x^{-r-2})$; also
$\alpha_0=O(x^{r-1})$ and $\beta_0=O(x^r)$, so $v_0\in\mathcal D_0(R)$.
For $v$ near $v_0$ and $\sigma$ near $\sigma_0$, write
\[
 a=\sigma+v/x,\qquad p=\sigma+v',\qquad d=xv'-v.
\]
The functions $a$ and $p$ stay in fixed complex neighborhoods of the
positive real number $\sigma_0$, and $d=O(x^{-r})$.
The function $a^{-1}$ and the branch of $(1+p^2)^{(r+1)/2}$ that is
positive for real $p$ are holomorphic on these neighborhoods.
Moreover,
\[
 F(x,\sigma x+v,\sigma+v')
 =\frac{a_0(1+p^2)}{xa}
  +\kappa x^{r-1}a^{r-1}(1+p^2)^{(r+1)/2}d.
\]
The weighted product estimates show that this expression is holomorphic
with values in $\mathcal C^0_1(R)$. The second derivative causes no
loss here: on the fixed graph domain,
\[
 v''=L_0v+\alpha_0v+\beta_0v'
\]
is a bounded linear map to $\mathcal C^0_1(R)$.
This proves the assertion.
\end{proof}

At $(\sigma_0,v_0)$ the derivative $D_v\mathscr G$ is $L_0$.
The analytic implicit-function theorem and
Lemma~\ref{lem:slope-linear-inverse} therefore give a holomorphic
family $v_\sigma$ in $\mathcal D_0(R)$ for $\sigma$ in a complex disc
about $\sigma_0$. For real $\sigma$ in this disc,
$\sigma x+v_\sigma$ is positive and has limiting slope $\sigma$.
Theorem~\ref{thm:uniqueness} identifies it with $u_\sigma$.
In particular, parameter analyticity is already established in
$\mathcal C^1_r(R)$, without a two-derivative Schauder estimate.

\subsection{Higher weighted derivative orders}

\begin{lemma}\label{lem:uniform-holomorphic-expansion}
The holomorphic family just constructed is locally bounded in
$\mathcal X^{k,\alpha}_r(R)$ for every $k\ge0$ and $0<\alpha<1$.
\end{lemma}

\begin{proof}
Work on a closed complex disc strictly inside the disc furnished by
the implicit-function theorem, reducing the outer disc if necessary.
The weighted $C^1$ bounds give uniformly
\[
 u_\sigma=\sigma x+O(x^{-r}),\quad
 u_\sigma'=\sigma+O(x^{-r-1}),\quad
 xu_\sigma'-u_\sigma=O(x^{-r}).
\]
For each fixed $N$, the coefficients $c_j(\sigma)$ in
\eqref{eq:coefficient-recursion} are holomorphic near $\sigma_0$;
their denominators do not vanish there.
The proof of Theorem~\ref{thm:expansion-intro} now applies uniformly
to the complex family, with the following changes made explicit.
Linearization along the segment joining $u_\sigma$ to $U_N$ uses the
same holomorphic branch of $F$. The cancellation
\eqref{eq:uniqueness-cancellation} still holds, and, on a sufficiently
distant half-line (depending on $N$),
\[
 \operatorname{Re}P_N(x)\ge \underline c\,x^r,\qquad
 |Q_N(x)|\le C_Nx^{-2},\qquad
 |E_N(x)|\le C_Nx^{-1-N(r+1)}.
\]
Here $\underline c>0$ follows by continuity from
$\kappa\sigma_0^{r-1}(1+\sigma_0^2)^{(r+1)/2}>0$.
Thus $|e^{-\int_x^sP_N}|\le e^{-\underline c\,x^r(s-x)}$, and the two integral
estimates for the remainder, including the absorption of $M_R$, are
unchanged. They yield the $C^0$ and $C^1$ remainder bounds uniformly
on the closed parameter disc.

Repeated differentiation of the ODE gives uniform polynomial bounds
for higher $x$-derivatives: $u_\sigma/x$ stays away from zero and
$u_\sigma'$ stays in the chosen holomorphic neighborhood. Applying
the local interpolation argument in the expansion proof to real and
imaginary parts gives its remainder estimates in every derivative
order, uniformly on the parameter disc. In particular,
\[
 |\partial_x^j(u_\sigma-\sigma x)|\le C_jx^{-r-j}.
\]
On the remaining compact interval these estimates follow from the
ODE and the uniform weighted $C^1$ bound.
The bound for the next derivative implies, on each dyadic interval,
\[
 [\partial_x^k(u_\sigma-\sigma x)]_{C^{0,\alpha}([x,2x])}
 \le C_{k,\alpha}x^{-r-k-\alpha}.
\]
This proves the claimed local boundedness in all the weighted spaces.
\end{proof}

\begin{theorem}\label{thm:analytic-slope-dependence}
For every compact interval $I\Subset(0,\infty)$ there is $R_I>0$
such that, for all $k\ge0$ and $0<\alpha<1$, the map
\[
 \sigma\longmapsto u_\sigma-\sigma x
 \quad\hbox{with values in }\mathcal X^{k,\alpha}_r(R_I)
\]
is real analytic on an open neighborhood of $I$.
For $0<\varepsilon<E<\infty$, the map
\[
 (0,\infty)\ni\sigma\longmapsto u_\sigma
 \in C^{k,\alpha}([\varepsilon,E])
\]
is also real analytic.
\end{theorem}

\begin{proof}
Fix $\sigma_0$ and the holomorphic family in the graph domain.
On a parameter circle of radius $d$ contained in its disc,
Lemma~\ref{lem:uniform-holomorphic-expansion} gives a bound $M_{k,\alpha}$
in $\mathcal X^{k,\alpha}_r(R)$.
Pointwise Cauchy integration defines the Taylor coefficients $a_j(x)$
of $v_\sigma(x)$. Differentiation in $x$ under this integral and the
same estimate for difference quotients give
\[
 \|a_j\|_{\mathcal X^{k,\alpha}_r(R)}\le M_{k,\alpha}d^{-j}.
\]
The Taylor series therefore converges in this Banach space whenever
$|\sigma-\sigma_0|<d$, and its pointwise sum is $v_\sigma$.
This proves analyticity in every weighted derivative order. For
$\sigma$ in a compact interval $I$ we take $R_I$ as the maximum of the
radii in a finite cover of $I$; restriction to a larger half-line
preserves all assertions.

Evaluation of $u_\sigma$ and $u_\sigma'$ at a fixed large radius is
analytic. The first-order system $u'=p$, $p'=F(x,u,p)$ is real analytic
for $u>0$. On $[\varepsilon,R_I]$ the real solutions have a uniform
positive lower bound. Analytic dependence on initial data therefore
propagates the parameter analyticity to this compact interval.
\end{proof}

\begin{corollary}\label{cor:analytic-exterior-foliation}
For every compact interval $I\Subset(0,\infty)$ there is $R_I>0$
such that
\[
 \partial_\sigma u_\sigma(x)=x+O_I(x^{-r})>0
 \qquad(\sigma\in I,\ x\ge R_I).
\]
For slopes in the interior of $I$, the rotational ends form a
real-analytic foliation of their image in $\{x>R_I\}$.
The same family extends analytically to a neighborhood of $I$.
\end{corollary}

\begin{proof}
Differentiate the analytic map with values in $\mathcal C^0_r(R_I)$.
Its derivative is uniformly bounded on $I$, which gives the stated
estimate and, after increasing $R_I$, the lower bound $x/2$.
The ODE is analytic in both $x$ and its initial data, so
$(x,\sigma)\mapsto u_\sigma(x)$ is jointly real analytic.
The map $(x,\sigma,\omega)\mapsto(x,u_\sigma(x)\omega)$ has nonzero
Jacobian in the slope direction and is injective on the indicated
region. The analytic inverse-function theorem gives a foliation of
its image; no claim of coverage of the entire ambient exterior is made.
\end{proof}

\begin{remark}
Analyticity and the exterior foliation are statements on distant
exterior regions. They do not imply that $u_\sigma(x)<u_{\sigma'}(x)$
for every $x>0$ when $\sigma<\sigma'$, and numerically this global
ordering appears to fail in the computed examples; see
Remark~\ref{rem:numerical-crossing}.
\end{remark}

\section{Uniqueness without rotational symmetry}
\label{sec:nonrotational}

The uniqueness established above concerns rotational profiles. We now
remove the symmetry assumption on the competing end, while retaining a
circular asymptotic cone. The proof uses the equation to obtain decay of
the normal graph and its difference from another solution, followed by
Bernstein's continuation theorem at infinity \cite{Bernstein2020}.
The argument also applies to certain noncircular cones, as explained
at the end of the section.

Throughout this section, $A=-dN$, $S_r$ is unnormalized, and
$\gamma=-c>0$. Thus the geometric equation is
\begin{equation}\label{eq:nr-shrinker}
 S_r(A)=-\gamma\ip{X}{N}.
\end{equation}
If the profile equation is written with coefficient $\kappa$, then
$\gamma=\binom{n-1}{r-1}\kappa$.

Let
\[
 C=\{\rho\theta:\rho>0,\ \theta\in\Gamma\}
\]
be an oriented regular cone, where $\Gamma\subset\Sph^n$ is a smooth
closed embedded hypersurface. Write
\[
 g_C=d\rho^2+\rho^2h,\qquad
 A_C\partial_\rho=0,\qquad
 A_C|_{T\Gamma}=\rho^{-1}K,
\]
where $K$ is the shape operator of the link. As in \eqref{eq:nr-graph-intro}, by a smoothly asymptotic
normal graph over $C$ we mean a parametrization
\begin{equation}\label{eq:nr-graph}
 X_f(\rho,\theta)=\rho\theta+f(\rho,\theta)N_C(\theta),
 \qquad
 |\nabla_C^j f|=o(\rho^{1-j})
 \quad(j\geq0).
\end{equation}
Smooth conical convergence with multiplicity one gives this description
after restricting the end. The orientation of the graph is chosen to
converge to $N_C$.

We use $O_\infty(\rho^{-a})$ for a function or tensor whose $j$th
covariant derivative has norm $O(\rho^{-a-j})$ for every $j$.
Equivalently, all conormal derivatives, generated by
$\rho\partial_\rho$ and smooth vector fields on the link, satisfy
the corresponding weighted bounds.

We prove Theorem~\ref{thm:nr-circular} below. The conclusion concerns
equality outside a compact set and does not identify arbitrary interior
extensions. No quantitative rate beyond \eqref{eq:nr-graph} is assumed.

\subsection{Normal graphs and automatic decay}

Put
\[
 B=I-fA_C,\qquad
 Z=B^{-1}\nabla_C f,\qquad
 W_f=\sqrt{1+|Z|^2}.
\]
On a sufficiently distant part of the cone, $B$ is invertible and
\[
 dX_f(U)=BU+df(U)N_C,\qquad
 N_f=\frac{N_C-Z}{W_f},
\]
\[
 g_f(U,V)=g_C(BU,BV)+df(U)df(V).
\]
Since $B\partial_\rho=\partial_\rho$, we obtain the exact identity
\[
 \ip{X_f}{N_f}=\frac{f-\rho f_\rho}{W_f}.
\]
Define
\(
 \mathcal E(f)=W_fS_r(A_f).
\)
Equation \eqref{eq:nr-shrinker} becomes
\begin{equation}\label{eq:nr-radial-identity}
 \gamma(\rho f_\rho-f)=\mathcal E(f).
\end{equation}

For later estimates, the exact second fundamental form is
\begin{equation}\label{eq:nr-second-form}
\begin{aligned}
 W_f b_f(U,V)={}&b_C(U,V)-f\,g_C(A_CU,A_CV)
                  +\nabla_C^2f(U,V)\\
 &+df(U)g_C(A_CV,Z)+df(V)g_C(A_CU,Z)\\
 &+f\,g_C((\nabla_UA_C)V,Z).
\end{aligned}
\end{equation}
Indeed, after subtracting $dX_f(\nabla_UV)$ from $D_UdX_f(V)$,
its normal component relative to the cone is
\[
 b_C(U,BV)+\nabla_C^2f(U,V),
\]
and its tangential component is
\[
 -df(U)A_CV-f(\nabla_UA_C)V-df(V)A_CU.
\]
Taking the product with $N_f$ proves
\eqref{eq:nr-second-form}. In particular, with our sign convention,
\begin{equation}\label{eq:nr-linearization-zero}
 D\mathcal E(0)[v]
 =
 \operatorname{tr}\bigl(T_{r-1}(A_C)\nabla_C^2v\bigr)
 +
 \operatorname{tr}\bigl(T_{r-1}(A_C)A_C^2\bigr)v,
\end{equation}
where $T_{r-1}$ is the Newton tensor.

\begin{lemma}\label{lem:nr-decay}
For a regular cone $C$, every solution satisfying
\eqref{eq:nr-graph} obeys
\begin{equation}\label{eq:nr-leading-graph}
 f=-\frac{S_r(K)}{\gamma(r+1)}\rho^{-r}
       +O_\infty(\rho^{-2r-1}).
\end{equation}
If $f_1,f_2$ are two such solutions over the same oriented cone, then
\begin{equation}\label{eq:nr-flat}
 f_1-f_2=O_\infty(\rho^{-M})
 \qquad\text{for every }M>0.
\end{equation}
\end{lemma}

\begin{proof}
Let $\mathscr D$ denote the algebra of conormal differential operators
generated by $\rho\partial_\rho$ and by smooth vector fields on the
link $\Gamma$. Thus, for a scalar function $q$,
\(
q=O_\infty(\rho^{-\alpha})
\)
means that
\(
|\mathcal Dq|\le C_{\mathcal D}\rho^{-\alpha}
\)
for every $\mathcal D\in\mathscr D$. This formulation is equivalent
to the covariant-derivative definition of $O_\infty$ on the cone.

The graph condition \eqref{eq:nr-graph}, together with the exact graph
formulas, implies
\[
\mathcal E(f)=O_\infty(\rho^{-r}).
\]
Indeed, on every rescaled annulus $\{\rho\le |X|\le 2\rho\}$, the
quantity $\rho^r\mathcal E(f)$ is a smooth function of $\theta$ and of
the dimensionless graph jets
\[
\frac{f}{\rho},\qquad
\nabla_C f,\qquad
\rho\nabla_C^2f,
\]
and their conormal derivatives. These jets are bounded and converge
to zero, with all conormal derivatives, by
\eqref{eq:nr-graph}.

Dividing \eqref{eq:nr-radial-identity} by $\rho^2$ gives
\[
\partial_\rho\left(\frac{f}{\rho}\right)
=
\frac{\mathcal E(f)}{\gamma\rho^2}.
\]
Since $f/\rho\to0$ as $\rho\to\infty$, integration from $\rho$ to
infinity yields
\begin{equation}\label{eq:nr-integral}
 f(\rho,\theta)
 =
 -\frac{\rho}{\gamma}
 \int_\rho^\infty
 \frac{\mathcal E(f)(t,\theta)}{t^2}\,dt.
\end{equation}
The integral is absolutely convergent because
$\mathcal E(f)=O_\infty(\rho^{-r})$ and $r\ge1$.

We record the elementary mapping property that will be used repeatedly:
if $q=O_\infty(\rho^{-\alpha})$, with $\alpha>-1$, then
\[
\mathcal Iq(\rho,\theta)
:=
\rho\int_\rho^\infty\frac{q(t,\theta)}{t^2}\,dt
=
O_\infty(\rho^{-\alpha}).
\]
For angular conormal fields this follows by differentiating under the
integral sign. For the radial conormal field one has the exact identity
\[
\rho\partial_\rho(\mathcal Iq)
=
\mathcal Iq-q.
\]
Iteration of this identity proves the assertion for arbitrary
conormal derivatives. Applying it to \eqref{eq:nr-integral} gives
\[
f=O_\infty(\rho^{-r}).
\]

We now improve this first estimate. On rescaled annuli,
$\mathcal E(f)$ has the form
\[
\mathcal E(f)
=
\rho^{-r}
\mathscr F\left(
 \theta,\frac{f}{\rho},\nabla_Cf,\rho\nabla_C^2f
\right),
\]
where $\mathscr F$ is smooth in its jet variables. The estimate just
obtained implies that every dimensionless graph jet is
$O_\infty(\rho^{-r-1})$. At the zero graph,
\[
\mathscr F(\theta,0,0,0)=S_r(K)(\theta).
\]
Taylor expansion in the jet variables, uniformly on the compact link,
therefore gives
\[
\mathcal E(f)
=
\rho^{-r}S_r(K)
+
O_\infty(\rho^{-2r-1}).
\]
Substituting this expansion into \eqref{eq:nr-integral} and using
\[
\rho\int_\rho^\infty t^{-r-2}\,dt
=
\frac{1}{r+1}\rho^{-r},
\]
we obtain
\[
f(\rho,\theta)
=
-\frac{S_r(K)(\theta)}{\gamma(r+1)}\rho^{-r}
+
O_\infty(\rho^{-2r-1}),
\]
which is \eqref{eq:nr-leading-graph}.

Let now $f_1$ and $f_2$ be two solutions over the same oriented cone,
and set
\[
w=f_1-f_2,
\qquad
f_\tau=f_2+\tau w,
\qquad 0\le\tau\le1.
\]
Subtracting their equations and applying the fundamental theorem of
calculus in the graph variable gives the exact equation
\begin{equation}\label{eq:nr-difference}
 \gamma(\rho w_\rho-w)=\mathcal Pw,
 \qquad
 \mathcal P=\int_0^1D\mathcal E(f_\tau)\,d\tau.
\end{equation}
In covariant form,
\[
\mathcal Pw
=
a^{ij}\nabla_{C,i}\nabla_{C,j}w
+b^i\nabla_{C,i}w
+dw,
\]
where, with all conormal derivatives,
\[
a=O_\infty(\rho^{1-r}),
\qquad
b=O_\infty(\rho^{-r}),
\qquad
d=O_\infty(\rho^{-r-1}).
\]
These estimates follow from the smooth dependence of $\mathcal E$ on
the dimensionless graph jets and from the uniform estimates for the
segment $\{f_\tau:0\le\tau\le1\}$.

So, for every $\alpha>0$,
\begin{equation}\label{eq:nr-decay-improvement}
w=O_\infty(\rho^{-\alpha})
\quad\text{implies}\quad
\mathcal Pw
=
O_\infty(\rho^{-\alpha-r-1}).
\end{equation}
Here the two covariant derivatives in the principal term contribute
$\rho^{-\alpha-2}$, which, combined with
$a=O_\infty(\rho^{1-r})$, gives precisely
$\rho^{-\alpha-r-1}$. The first- and zeroth-order terms have the same
decay.

As before, the condition $w/\rho\to0$ allows us to integrate
\eqref{eq:nr-difference}:
\[
w(\rho,\theta)
=
-\frac{\rho}{\gamma}
\int_\rho^\infty
\frac{\mathcal Pw(t,\theta)}{t^2}\,dt.
\]
The mapping property of $\mathcal I$ and
\eqref{eq:nr-decay-improvement} therefore show that
\[
w=O_\infty(\rho^{-\alpha})
\quad\text{implies}\quad
w=O_\infty(\rho^{-\alpha-r-1}).
\]
Initially, $w=O_\infty(\rho^{-r})$ because both $f_1$ and $f_2$ have
that decay. In fact, the common leading term already gives
$w=O_\infty(\rho^{-2r-1})$, but the weaker initial estimate suffices.
Iterating the preceding implication yields
\[
w=O_\infty\bigl(\rho^{-r-j(r+1)}\bigr)
\qquad\text{for every }j\ge0.
\]
Given any $M>0$, choosing $j$ sufficiently large proves
\[
f_1-f_2=w=O_\infty(\rho^{-M}),
\]
which is \eqref{eq:nr-flat}.
\end{proof}

Notice that flatness at infinity alone does not imply equality.
The remaining argument supplies the required continuation principle.

\subsection{Bernstein's theorem on weakly conical ends}
We recall the result that provides this principle. Following
\cite[Section~2, p.~1899]{Bernstein2020}, a triple $(E,g,\mathcal R)$,
where $E$ is a smooth $n$-dimensional manifold with $n\ge2$, $g$ a $C^1$
Riemannian metric on $E$, and $\mathcal R:E\to(\mathcal R_E,\infty)$, with
$\mathcal R_E\ge1$, a proper unbounded $C^2$ function, is a
\emph{weakly conical end} if there is $\Lambda\ge0$ such that
\begin{equation}\label{eq:weakly-conical}
\bigl||\nabla_g\mathcal R|-1\bigr|\le\frac{\Lambda}{\mathcal R^{4}}\le\frac12,
\qquad
\bigl|\nabla_g^2\mathcal R^2-2g\bigr|\le\frac{\Lambda}{\mathcal R^{2}}\le\frac12 .
\end{equation}
Write $\partial_{\mathcal R}=\nabla_g\mathcal R$ and
$\mathcal L_0=\Delta_g-\frac{\mathcal R}{2}\partial_{\mathcal R}$.

\begin{theorem}[{Bernstein \cite[Theorem~1.2, p.~1898]{Bernstein2020}}]
\label{thm:bernstein}
Let $(E,g,\mathcal R)$ be a weakly conical end and let $u\in C^2(E)$ satisfy,
for some constants $\varpi\in\R$ and $M\ge0$,
\[
\bigl|(\mathcal L_0+\varpi)u\bigr|\le M\mathcal R^{-2}\bigl(|u|+|\nabla_gu|\bigr)
\quad\text{and}\quad
\int_E\bigl(|\nabla_gu|^2+u^2\bigr)\mathcal R^{2-4\varpi}
e^{-\mathcal R^2/4}<\infty .
\]
Then there are constants $R_0$ and $K_0$, depending on $u$, such that for
every $R\ge R_0$
\[
\int_{\{\mathcal R\ge R\}}\Bigl(u^2+\mathcal R^2|\nabla_gu|^2
+\mathcal R^4\bigl(\partial_{\mathcal R}u-\tfrac{2\varpi}{\mathcal R}u\bigr)^2\Bigr)
\mathcal R^{-1-n-4\varpi}
\le\frac{K_0}{R^{n+4\varpi}}\int_{\{\mathcal R=R\}}u^2 .
\]
Moreover, $u$ is asymptotically homogeneous of degree $2\varpi$, with
asymptotic trace $a\in H^1(L(E))$ on the link $L(E)$ of the asymptotic
cone, and
\[
\alpha^2:=\lim_{t\to\infty}t^{1-n-4\varpi}\int_{\{\mathcal R=t\}}u^2
=\int_{L(E)}a^2 .
\]
Finally, if $\mathcal U$ denotes the leading term of $u$, then for $R\ge R_0$
\[
\int_{\{\mathcal R\ge R\}}\Bigl(u^2+\mathcal R^2\bigl(|u-\mathcal U|^2+|\nabla_gu|^2\bigr)
+\mathcal R^4\bigl(\partial_{\mathcal R}u-\tfrac{2\varpi}{\mathcal R}u\bigr)^2\Bigr)
\mathcal R^{-2-n-4\varpi}
\le\frac{K_0\,\alpha^2}{R^{2}} .
\]
\end{theorem}

In particular, if $\alpha=0$, the last estimate forces $u\equiv0$ on
$\{\mathcal R\ge R_0\}$. This is the unique continuation principle at
infinity that we need; it is also the way in which Bernstein deduces the
uniqueness of asymptotically conical self-shrinkers
\cite[Theorem~1.1, p.~1897, and Section~8, pp.~1920--1921]{Bernstein2020}.

\subsection{The effective metric for a circular cone}

For $C=\calC_\sigma$, the principal curvatures of its link equal
$k=\sigma^{-1}$. Define
\[
 a_R=\binom{n-1}{r-1}k^{r-1},\qquad
 a_T=\binom{n-2}{r-1}k^{r-1},\qquad
 \beta=\frac{r+1}{2}.
\]
Notice that $2\beta=r+1=q$, where $q$ is the asymptotic
reparametrization exponent introduced in Section~\ref{sec:expansion}.
Thus the radial change of variables below is governed by the same
homogeneity that determines the powers in the profile expansion.
Both constants are positive in the range $1\le r\le n-1$, and
\[
 \frac{a_T}{a_R}=\frac{n-r}{n-1}.
\]
Equation \eqref{eq:nr-linearization-zero} reads
\begin{equation}\label{eq:nr-cone-operator}
\begin{aligned}
 D\mathcal E(0)[w]={}&
 a_R\rho^{1-r}w_{\rho\rho}
 +(n-1)a_T\rho^{-r}w_\rho
 +a_T\rho^{-r-1}\Delta_h w\\
 &+(n-1)a_Tk^2\rho^{-r-1}w.
\end{aligned}
\end{equation}
In particular, the scaled principal symbol is positive definite.
The same holds for $\mathcal P$ sufficiently far out.

We first record a tensorial identity for links with constant principal
curvatures.

\begin{lemma}\label{lem:spectral-cancellation}
Let $\Gamma\subset\Sph^n$ be a smooth hypersurface whose principal
curvatures, with their multiplicities, are constant, and let $K$ be its
shape operator. Then, for every polynomial $p$ and every vector $Y$
tangent to $\Gamma$,
\[
 \operatorname{tr}\bigl(p(K)\nabla_YK\bigr)=0,
 \qquad
 \sum_i p(k_i)\,\ip{(\nabla_{e_i}K)e_i}{Y}_h=0,
\]
where $(e_i)$ is a local $h$-orthonormal frame of principal directions.
\end{lemma}

\begin{proof}
Let $\Phi$ be a polynomial with $\Phi'=p$. Then
$\operatorname{tr}(p(K)\nabla_YK)=Y\bigl(\operatorname{tr}\Phi(K)\bigr)$,
and $\operatorname{tr}\Phi(K)=\sum_i\Phi(k_i)$ is constant on $\Gamma$.
For the second identity, $\nabla_{e_i}K$ is symmetric, and the Codazzi
equation of $\Gamma$ in the unit sphere gives
$\ip{(\nabla_{e_i}K)e_i}{Y}_h=\ip{(\nabla_{e_i}K)Y}{e_i}_h
=\ip{(\nabla_YK)e_i}{e_i}_h$. Summing with the weights $p(k_i)$ yields
$\operatorname{tr}(p(K)\nabla_YK)=0$.
\end{proof}

The next lemma makes precise in which sense the operator $\mathcal P$
of \eqref{eq:nr-difference} is a small perturbation of the
linearization at a radial function. Let $e_\rho=\partial_\rho$ and let
$e_a=\rho^{-1}\epsilon_a$, where $(\epsilon_a)$ is a local
$h$-orthonormal frame on the link, so that $(e_\rho,e_a)$ is
$g_C$-orthonormal. Write
\[
 \mathcal P w
 =
 \mathsf a^{\rho\rho}\nabla^2_{\rho\rho}w
 +2\mathsf a^{\rho a}\nabla^2_{\rho a}w
 +\mathsf a^{ab}\nabla^2_{ab}w
 +\mathsf b^\rho\nabla_\rho w+\mathsf b^a\nabla_aw+\mathsf d\,w
\]
in this frame, with covariant derivatives of $g_C$.

\begin{lemma}\label{lem:nr-coefficients}
Let $C$ be the cone over a link $\Gamma$ with constant principal
curvatures, let $F=-\Theta\rho^{-r}$ with
$\Theta=S_r(K)/(\gamma(r+1))$, and let $f_t=f_2+tw$ as above. Then:
\begin{enumerate}
\item at $F$, the mixed coefficient and the angular first-order
coefficient of $D\mathcal E(F)$ vanish. Its radial and zeroth-order
coefficients depend only on $\rho$, while its angular principal block
is a spectral function of $K$ with scalar coefficients depending only
on $\rho$. In a local principal frame its entries depend only on
$\rho$; this does not imply that $K$ is parallel or that the entries
are radial in an arbitrary orthonormal frame;
\item each coefficient of $\mathcal P$ differs from the corresponding
coefficient of $D\mathcal E(F)$ by $\rho^{1-r}O_\infty(\rho^{-2r-2})$ for
the second-order coefficients, by $\rho^{-r}O_\infty(\rho^{-2r-2})$ for
the first-order coefficients, and by $\rho^{-r-1}O_\infty(\rho^{-2r-2})$
for the zeroth-order coefficient.
\end{enumerate}
\end{lemma}

\begin{proof}
On a rescaled annulus $\{\rho\le|X|\le2\rho\}$, $\rho^{r-1}$ times the
second-order coefficients of $D\mathcal E(f)$, $\rho^{r}$ times its
first-order coefficients, and $\rho^{r+1}$ times its zeroth-order
coefficient are smooth functions of $\theta$ and of the dimensionless
jets $\mathcal J(f)=(f/\rho,\nabla_Cf,\rho\nabla_C^2f)$; indeed,
$\mathcal E$ is a second-order operator in $f$, homogeneous under the
dilations of the cone.

(i) For the radial function $F$, the graph metric, its shape operator
and its Newton tensor are block diagonal with respect to the splitting
into the radial direction and $T\Gamma$, and the angular blocks are
spectral functions of $K$ with coefficients depending only on $\rho$.
Hence the mixed coefficient vanishes. Every term of the angular
first-order coefficient contains either an angular derivative of $F$,
which vanishes, or a contraction of $\nabla K$ with spectral functions of
$K$, which vanishes by Lemma~\ref{lem:spectral-cancellation}. For the circular cone this
also follows from $O(n)$-invariance: an $O(n)$-invariant vector field on
$\Sph^{n-1}$ vanishes.

(ii) By \eqref{eq:nr-leading-graph},
$f_i=F+O_\infty(\rho^{-2r-1})$ for $i=1,2$, hence
$\mathcal J(f_t)-\mathcal J(F)=O_\infty(\rho^{-2r-2})$ uniformly in
$t\in[0,1]$. Since the normalized coefficients are smooth functions of
the jets, Taylor expansion around $\mathcal J(F)$ and integration in $t$
give the assertion.
\end{proof}

\begin{lemma}\label{lem:nr-normal-form}
For the difference of two solutions over $\calC_\sigma$, there are
a smooth metric $\widehat g$, a proper radial function $\mathcal R$, and
constants $\ell,\varpi_*$ such that $v=\mathcal R^{\ell/2}w$ satisfies
\begin{equation}\label{eq:nr-bernstein-inequality}
 \left|
 (\Delta_{\widehat g}+\varpi_*)v
 -\frac{\mathcal R}{2}
 \ip{\nabla_{\widehat g}\mathcal R}{\nabla_{\widehat g}v}
 \right|
 \leq
 C \mathcal R^{-2}\bigl(|v|+|\nabla_{\widehat g}v|\bigr).
\end{equation}
Moreover,
\begin{equation}\label{eq:nr-weak-cone}
 |\nabla_{\widehat g}\mathcal R|=1,\qquad
 |\nabla_{\widehat g}^2\mathcal R^2-2\widehat g|
 \leq C\mathcal R^{-2},
\end{equation}
and $v$ and all its derivatives decay faster than every power of $\mathcal R$.
\end{lemma}

\begin{proof}
By Lemma~\ref{lem:nr-decay}, the two graph functions have the common
leading term
\[
 f_i=F+O_\infty(\rho^{-2r-1}),
 \qquad
 F=-\Theta\rho^{-r},
 \qquad
 \Theta=\frac{\binom{n-1}{r}k^r}{\gamma(r+1)}.
\]
Therefore, for
\[
 f_t=f_2+tw,\qquad 0\le t\le1,
\]
we have, uniformly in \(t\),
\begin{equation}\label{eq:nr-ft-F}
 \frac{f_t-F}{\rho}
 =
 O_\infty(\rho^{-2r-2}).
\end{equation}

For the radial graph determined by \(F\), the meridian and angular
principal curvatures are
\[
 \mu=
 \frac{F''}{(1+F'^2)^{3/2}},
 \qquad
 \lambda=
 \frac{k+F'}
 {\rho(1-kF/\rho)\sqrt{1+F'^2}}.
\]
Writing \(W_F=\sqrt{1+F'^2}\), we obtain
\[
 \mathcal E(F)
 =
 \binom{n-1}{r-1}
 \lambda^{r-1}\frac{F''}{W_F^2}
 +
 \binom{n-1}{r}W_F\lambda^r.
\]
Recall that
\[
 a_R=\binom{n-1}{r-1}k^{r-1},
 \qquad
 a_T=\binom{n-2}{r-1}k^{r-1},
 \qquad
 \frac{a_T}{a_R}=\frac{n-r}{n-1}.
\]
Since
\[
 F'=r\Theta\rho^{-r-1},
 \qquad
 \frac{kF}{\rho}=-k\Theta\rho^{-r-1},
\]
expansion of the preceding formulas gives the coefficient of
\(w_{\rho\rho}\) in \(D\mathcal E(F)\) as
\begin{equation}\label{eq:nr-radial-principal}
 a_R\rho^{1-r}
 \left(
  1+\zeta\rho^{-r-1}
  +O_\infty(\rho^{-2r-2})
 \right),
 \qquad
 \zeta=(r-1)\Theta(r/k-k).
\end{equation}
The coefficient of \(w_\rho\) has leading term
\begin{equation}\label{eq:nr-radial-first}
 (n-r)a_R\rho^{-r}.
\end{equation}
Moreover, the angular principal part is
\[
 a_T\rho^{-r-1}\Delta_h
\]
to leading order.

By Lemma~\ref{lem:nr-coefficients}, replacing \(D\mathcal E(F)\) by
\[
 \mathcal P=\int_0^1D\mathcal E(f_t)\,dt
\]
changes each of its coefficients by a relative
\(O_\infty(\rho^{-2r-2})=O_\infty(s^{-4})\), measured in the natural
scale of that coefficient.

Set
\[
 \beta=\frac{r+1}{2},
 \qquad
 s=\rho^\beta,
 \qquad
 A_0=a_R\beta^2,
 \qquad
 \eta=\frac{\gamma\beta}{A_0}.
\]
Since
\[
 \rho\partial_\rho=\beta s\partial_s,
 \qquad
 \rho^2\partial_{\rho\rho}
 =
 \beta^2s^2\partial_{ss}
 +\beta(\beta-1)s\partial_s,
\]
the equation
\[
 \mathcal Pw-\gamma\rho w_\rho+\gamma w=0
\]
takes, in the variables \((s,\theta)\), the form
\begin{equation}\label{eq:nr-s-operator}
 A\,w_{ss}
 +2A^a w_{sa}
 +A^{ab}w_{ab}
 +B\,w_s+B^aw_a+Cw=0.
\end{equation}
Here $w_s=\partial_sw$, $w_{ss}=\partial_s^2w$,
$w_a=\partial_aw$, $w_{sa}=\partial_s\partial_aw$, and
$w_{ab}=\nabla_a^h\nabla_b^hw$.
Thus the angular second derivatives are covariant with respect to $h$.
In ordinary coordinates the angular first-order coefficient is
\[
 B^a_{\rm coord}=B^a-A^{bc}\Gamma^a_{bc}(h).
\]
The connection term in this identity is retained when comparing the
operator with a Laplace--Beltrami operator.
The coefficients
satisfy, with all derivatives measured on the conical scale,
\begin{align}
 A(s,\theta)
 &=
 A_0\left(
  1+\zeta s^{-2}+O_\infty(s^{-4})
 \right),                                         \label{eq:nr-A}\\
 B(s,\theta)
 &=
 -\gamma\beta s
 +a_R\beta(\beta-1+n-r)s^{-1}
 +O_\infty(s^{-3}),                               \label{eq:nr-B}\\
 C(s,\theta)
 &=
 \gamma+O_\infty(s^{-2}).                         \label{eq:nr-C}
\end{align}
Furthermore,
\begin{align}
 A^{ab}
 &=
 a_Ts^{-2}
 \left(
  h^{ab}+s^{-2}Y^{ab}+O_\infty(s^{-4})
 \right),                                         \label{eq:nr-angular}\\
 A^a&=O_\infty(s^{-5})A_0,                        \label{eq:nr-mixed}
\end{align}
where $Y^{ab}$ is a constant multiple of $h^{ab}$ and is
independent of $s$. These expansions, as well as
\eqref{eq:nr-A}--\eqref{eq:nr-C}, follow from the radial computations
above and Lemma~\ref{lem:nr-coefficients}. For instance, by
Lemma~\ref{lem:nr-coefficients}(i) the mixed term of \(\mathcal P\) is
\(2\mathsf a^{\rho a}\nabla^2_{\rho a}w\) with
\(\mathsf a^{\rho a}=\rho^{1-r}O_\infty(\rho^{-2r-2})\); since
\(\nabla^2_{\rho a}=\rho^{-1}\partial_\rho\epsilon_a\) up to lower-order
terms and \(\rho\partial_\rho=\beta s\partial_s\), its coefficient in
\eqref{eq:nr-s-operator} is
\(\rho^{-r-1}\beta s\,O_\infty(\rho^{-2r-2})=O_\infty(s^{-5})\), which is
\eqref{eq:nr-mixed}. At the radial graph $F$ the angular first-order
coefficient vanishes after the angular Hessian has been written
covariantly (Lemma~\ref{lem:nr-coefficients}(i)); this does not mean
that the angular coordinate connection terms vanish. Its perturbation is
\(\mathsf b^a=\rho^{-r}O_\infty(\rho^{-2r-2})\), and the same conversion
gives the coordinate coefficient \(B^a=O_\infty(s^{-6})\). In
particular,
\begin{equation}\label{eq:nr-angular-first}
 |B^a|_{s^2h}=O(s^{-3}).
\end{equation}

For \(s\) sufficiently large, \(A>0\). Divide
\eqref{eq:nr-s-operator} by \(A\) and define a contravariant metric by
\[
 g^{ss}=1,\qquad
 g^{sa}=\frac{A^a}{A},\qquad
 g^{ab}=\frac{A^{ab}}{A}.
\]
Strict ellipticity of the Newton tensor implies that this matrix is
positive definite. Thus it is the inverse of a smooth Riemannian
metric \(g\). The principal symbol of the normalized operator is $g^{ij}$.
In ordinary coordinates its second-order part is $g^{ij}\partial_{ij}$,
and its angular first-order part includes
$-(A^{bc}/A)\Gamma^a_{bc}(h)\partial_a$ in addition to $B^a/A$.
At the radial graph $F$, the angular metric is a scalar multiple of
$h$ at each fixed $s$, so these connection terms agree exactly with
the angular first-order terms of the corresponding $\Delta_g$.
The nonradial perturbation has relative size $O_\infty(s^{-4})$;
its differentiated contribution to this comparison is of conical norm
$O(s^{-5})$, which is smaller than the error retained below.

Define the limiting conical metric
\[
 g_*
 =
 ds^2+h_{*\,ab}s^2d\theta^a d\theta^b,
 \qquad
 h_*=\frac{A_0}{a_T}h.
\]
Equations \eqref{eq:nr-A}--\eqref{eq:nr-mixed}, followed by inversion
of the coefficient matrix, give
\begin{align}
 g^{ss}&=1,                                                     \label{eq:nr-gss}\\
 g^{sa}&=O_\infty(s^{-5}),                                     \label{eq:nr-gsa}\\
 g^{ab}
 &=
 s^{-2}
 \left(
  h_*^{ab}+s^{-2}\widetilde h^{ab}+O_\infty(s^{-4})
 \right),                                                      \label{eq:nr-gab}
\end{align}
where \(\widetilde h^{ab}\) is independent of \(s\). Equivalently,
\[
 g
 =
 ds^2+s^2h_*+h_1+O_\infty(s^{-2}),
\]
up to the mixed terms in \eqref{eq:nr-gsa}, where \(h_1\) is a
symmetric tensor on the link independent of \(s\). In particular,
\[
 |g-g_*|_{g_*}=O(s^{-2})
\]
with the corresponding conormal estimates in all orders.

Since \(g^{ss}=1\), we have the exact identity
\[
 |\nabla_gs|_g=1.
\]
The metric expansion also gives
\[
 \partial_s\log\sqrt{\det g}
 =
 \frac{n-1}{s}+O(s^{-3}),
\]
and hence
\begin{equation}\label{eq:nr-lap-s}
 \Delta_gs
 =
 \frac{1}{\sqrt{\det g}}
 \partial_i\bigl(\sqrt{\det g}\,g^{is}\bigr)
 =
 \frac{n-1}{s}+O(s^{-3}).
\end{equation}
For completeness, the Hessian identity
\[
 \nabla_g^2s^2=2\,ds\otimes ds+s\mathcal L_{\nabla_gs}g
\]
also verifies the required tensor estimate directly.
For the radial part $ds^2+s^2h_*+h_1$, subtraction of $2g$ leaves
$-2h_1$ in the angular block, of conical norm $O(s^{-2})$.
The differentiated remainder and the mixed terms have the same or
smaller order. Hence
\[
 \nabla_g^2s^2
 =
 2\,ds\otimes ds+2s\nabla_g^2s
 =
 2g+O_{g_*}(s^{-2}).
\]
Thus
\begin{equation}\label{eq:nr-conical-s}
 |\nabla_gs|_g=1,
 \qquad
 |\nabla_g^2s^2-2g|_g=O(s^{-2}).
\end{equation}
Moreover, \eqref{eq:nr-gsa} implies
\begin{equation}\label{eq:nr-coordinate-gradient}
 |\nabla_gs-\partial_s|_g=O(s^{-4}).
\end{equation}

We next compare the normalized operator with \(\Delta_g\). In local
coordinates,
\[
 \Delta_gw
 =
 g^{ij}\partial_{ij}w
 +
 \left(
  \frac{1}{\sqrt{\det g}}
  \partial_i(\sqrt{\det g}\,g^{ij})
 \right)\partial_jw.
\]
Therefore the second-order terms cancel exactly when the normalized
version of \eqref{eq:nr-s-operator} is written in terms of
\(\Delta_g\). By \eqref{eq:nr-lap-s}, the radial first-order term
contributed by \(\Delta_g\) is
\[
 \frac{n-1}{s}\partial_sw+O(s^{-3})|\nabla_gw|.
\]
On the other hand, dividing \eqref{eq:nr-B} by
\eqref{eq:nr-A} gives
\begin{align*}
 \frac BA
 &=
 \frac{
 -\gamma\beta s+
 a_R\beta(\beta-1+n-r)s^{-1}
 +O_\infty(s^{-3})
 }{
 A_0(1+\zeta s^{-2}+O_\infty(s^{-4}))
 }\\
 &=
 -\eta s
 +
 \left(
  \eta\zeta+
  \frac{\beta-1+n-r}{\beta}
 \right)s^{-1}
 +O_\infty(s^{-3}).
\end{align*}
Subtracting the radial first-order term of $\Delta_g$, retaining the
angular connection cancellation described above, and using
\eqref{eq:nr-angular-first}, yields
\begin{equation}\label{eq:nr-weighted-drift}
 \Delta_gw-\eta s\partial_sw+\frac{\ell}{s}\partial_sw
 +\frac{\gamma}{A_0}w
 =
 E_1\cdot\nabla_gw+E_0w,
\end{equation}
where
\[
 |E_1|_g=O(s^{-3}),
 \qquad
 |E_0|=O(s^{-2}),
\]
and
\begin{align}
 \ell
 &=
 \eta\zeta
 +\frac{\beta-1+n-r}{\beta}
 -(n-1)                                           \label{eq:nr-m}\\
 &=
 \eta\zeta-\frac{n(r-1)}{r+1}.                   \notag
\end{align}
The term \(\eta\zeta/s\) comes from the product of the leading drift
\(-\gamma\beta s\) with the correction
\(-\zeta s^{-2}\) in \(A^{-1}\); it is therefore part of the principal
\(s^{-1}\)-drift and cannot be placed in the error.

Now put
\[
 \mathcal R=\sqrt{2\eta}\,s,
 \qquad
 \widehat g=2\eta g.
\]
Since constant rescaling gives
\[
 \Delta_{\widehat g}=\frac1{2\eta}\Delta_g,
 \qquad
 \nabla_{\widehat g}\mathcal R\cdot\nabla
 =
 \frac{1}{\sqrt{2\eta}}\nabla_gs\cdot\nabla,
\]
division of \eqref{eq:nr-weighted-drift} by \(2\eta\), together with
\eqref{eq:nr-coordinate-gradient}, gives
\begin{equation}\label{eq:nr-w-before-conjugation}
 \left(
  \Delta_{\widehat g}
  -\frac{\mathcal R}{2}\langle\nabla_{\widehat g}\mathcal R,\nabla_{\widehat g}\,\cdot\,\rangle
  +\frac{\ell}{\mathcal R}
   \langle\nabla_{\widehat g}\mathcal R,\nabla_{\widehat g}\,\cdot\,\rangle
  +\varpi_0
 \right)w
 =
 \widetilde E_1\cdot\nabla_{\widehat g}w
 +\widetilde E_0w,
\end{equation}
where
\[
 \varpi_0=\frac{\gamma}{2\eta A_0}
 =\frac1{r+1},
 \qquad
 |\widetilde E_1|_{\widehat g}=O(\mathcal R^{-3}),
 \qquad
 |\widetilde E_0|=O(\mathcal R^{-2}).
\]

Finally, set
\[
 w=\mathcal R^{-\ell/2}v.
\]
Because \(|\nabla_{\widehat g}\mathcal R|=1\), the first-order term produced by
the Laplacian is
\[
 -\frac{\ell}{\mathcal R}
 \langle\nabla_{\widehat g}\mathcal R,\nabla_{\widehat g}v\rangle,
\]
which cancels the \(\ell/\mathcal R\)-drift in
\eqref{eq:nr-w-before-conjugation}. The leading zeroth-order
contribution from the Gaussian drift is
\[
 -\frac{\mathcal R}{2}\,\partial_{\mathcal R}(\mathcal R^{-\ell/2})\,\mathcal R^{\ell/2}
 =
 \frac \ell4,
\]
whereas the remaining terms produced by the conjugation are
\(O(\mathcal R^{-2})v\). Therefore
\[
 \left|
 \left(
  \Delta_{\widehat g}
  -\frac{\mathcal R}{2}
   \langle\nabla_{\widehat g}\mathcal R,\nabla_{\widehat g}\,\cdot\,\rangle
  +\varpi_*
 \right)v
 \right|
 \le
 C\mathcal R^{-2}\bigl(|v|+|\nabla_{\widehat g}v|\bigr),
\]
with
\[
 \varpi_*=\frac1{r+1}+\frac \ell4.
\]
This is \eqref{eq:nr-bernstein-inequality}.

Finally, \eqref{eq:nr-conical-s} is invariant under the constant
rescaling above and gives
\[
 |\nabla_{\widehat g}\mathcal R|_{\widehat g}=1,
 \qquad
 |\nabla_{\widehat g}^2\mathcal R^2-2\widehat g|_{\widehat g}
 \le C\mathcal R^{-2},
\]
which is \eqref{eq:nr-weak-cone}. Since multiplication by the fixed
power \(\mathcal R^{\ell/2}\) preserves superalgebraic decay,
\eqref{eq:nr-flat} implies that \(v\), together with all its
derivatives, decays faster than every power of \(\mathcal R\).
\end{proof}

\begin{proof}[Proof of Theorem~\ref{thm:nr-circular}]
Let \(f_1,f_2\) be the normal graph functions of the two ends over
\(\calC_\sigma\), and let \(w=f_1-f_2\). By
Lemma~\ref{lem:nr-normal-form}, there are a smooth metric \(\widehat g\)
on a distant part \(E\) of the cone, a proper function \(\mathcal R\), and
constants \(\ell,\varpi_*\) such that \(v=\mathcal R^{\ell/2}w\) satisfies
\begin{equation}\label{eq:nr-Bernstein-application}
 \bigl|(\mathcal L_0+\varpi_*)v\bigr|
 \le
 C\mathcal R^{-2}\bigl(|v|+|\nabla_{\widehat g}v|\bigr),
 \qquad
 \mathcal L_0=\Delta_{\widehat g}-\frac{\mathcal R}{2}\partial_{\mathcal R},
\end{equation}
where \(\partial_{\mathcal R}=\nabla_{\widehat g}\mathcal R\), and
\[
 |\nabla_{\widehat g}\mathcal R|_{\widehat g}=1,
 \qquad
 \bigl|\nabla_{\widehat g}^2\mathcal R^2-2\widehat g\bigr|_{\widehat g}
 \le C\mathcal R^{-2}.
\]
After increasing the inner radius, so that \(C\mathcal R^{-2}\le1/2\),
these estimates show that \((E,\widehat g,\mathcal R)\) is a weakly
conical end in the sense of \eqref{eq:weakly-conical}. Since
\(\widehat g\) is uniformly equivalent to the conical metric
\(2\eta(ds^2+s^2h_*)\) and the link is compact, its level sets satisfy
\begin{equation}\label{eq:nr-level-volume}
 \operatorname{Vol}_{\widehat g}\{\mathcal R=t\}
 \le Ct^{n-1}
\end{equation}
for all sufficiently large \(t\), and the volume of
\(\{t\le\mathcal R\le t+1\}\) is \(O(t^{n-1})\).

By Lemma~\ref{lem:nr-normal-form}, \(v\) and all its derivatives decay
faster than every power of \(\mathcal R\): for every \(N>0\),
\(|v|+|\nabla_{\widehat g}v|\le C_N\mathcal R^{-N}\) on \(E\). Hence
\[
 \int_E
 \bigl(v^2+|\nabla_{\widehat g}v|^2\bigr)
 \mathcal R^{2-4\varpi_*}e^{-\mathcal R^2/4}\,d\mu_{\widehat g}
 <\infty,
\]
and \(v\) satisfies the hypotheses of Theorem~\ref{thm:bernstein} with
\(\varpi=\varpi_*\). Moreover, by \eqref{eq:nr-level-volume}, for every
\(N>0\),
\[
 t^{1-n-4\varpi_*}
 \int_{\{\mathcal R=t\}}v^2\,d\sigma_{\widehat g}
 \le
 C_N^2\,t^{1-n-4\varpi_*}\,t^{n-1}\,t^{-2N}
 =
 C_N^2\,t^{-4\varpi_*-2N},
\]
which tends to zero as \(t\to\infty\) once \(N>-2\varpi_*\). Thus the
number \(\alpha\) in Theorem~\ref{thm:bernstein} vanishes, and the
last estimate of that theorem gives \(v\equiv0\) on
\(\{\mathcal R\ge R_0\}\). Since \(\mathcal R^{\ell/2}>0\),
also \(w=\mathcal R^{-\ell/2}v=0\) there. Thus \(f_1=f_2\) on a
sufficiently distant part of the cone, and the two ends coincide outside a
compact set. The proof uses neither the rotational
construction nor Theorems~\ref{thm:main} and~\ref{thm:expansion-intro};
these enter only in Corollary~\ref{cor:classification}.
\end{proof}

\subsection{Cones with constant principal curvatures on the link}

\begin{proposition}\label{prop:nr-constant-curvatures}
Let $\Gamma\subset\Sph^n$ be a smooth closed embedded oriented
hypersurface with constant principal curvatures
$k_1,\ldots,k_{n-1}$ and constant multiplicities. Assume
\[
 (0,k_1,\ldots,k_{n-1})\in\Ga_r.
\]
Two solutions of \eqref{eq:nr-shrinker} that are smoothly asymptotic,
with multiplicity one and compatible orientation, to the cone
over $\Gamma$ coincide outside a compact set.
\end{proposition}

\begin{proof}
Lemma~\ref{lem:nr-decay} applies without change. Since the principal
curvatures \(k_1,\ldots,k_{n-1}\) of the link are constant,
\[
 \Theta=\frac{S_r(K)}{\gamma(r+1)}
\]
is constant on \(\Gamma\). Hence the common leading term of the two
normal graph functions is the radial function
\[
 F(\rho)=-\Theta\rho^{-r}.
\]
Set
\[
 a_R=S_{r-1}(K),
 \qquad
 T=T_{r-1}(K).
\]
Because
\(
 (0,k_1,\ldots,k_{n-1})\in\Gamma_r,
\)
the Newton tensor associated with the cone is positive definite.
In particular,
\[
 a_R>0,\qquad T>0.
\]
Newton's identities also give
\begin{equation}\label{eq:nr-general-trace}
 \operatorname{tr}_hT=(n-r)S_{r-1}(K)=(n-r)a_R.
\end{equation}

For the radial graph determined by \(F\), the radial principal
curvature and the angular principal curvatures are
\[
 \mu=
 \frac{F''}{(1+F'^2)^{3/2}},
 \qquad
 \lambda_i=
 \frac{k_i+F'}
 {\rho(1-k_iF/\rho)\sqrt{1+F'^2}}.
\]
The coefficient of the radial second derivative in
\(D\mathcal E(F)\) is
\[
 \frac{S_{r-1}(\lambda_1,\ldots,\lambda_{n-1})}
 {1+F'^2}.
\]
Using
\[
 F'=r\Theta\rho^{-r-1},
 \qquad
 \frac{F}{\rho}=-\Theta\rho^{-r-1},
\]
we obtain
\[
 \lambda_i
 =
 \rho^{-1}
 \left(
  k_i+\Theta(r-k_i^2)\rho^{-r-1}
  +O(\rho^{-2r-2})
 \right).
\]
Therefore
\[
 S_{r-1}(\lambda)
 =
 \rho^{1-r}
 \left[
  a_R
  +
  \Theta\sum_i
  S_{r-2}(k|i)(r-k_i^2)\rho^{-r-1}
  +O(\rho^{-2r-2})
 \right],
\]
where \(k|i\) denotes the list obtained by omitting \(k_i\).
Since \(F'^2=O(\rho^{-2r-2})\), the radial coefficient becomes
\[
 a_R\rho^{1-r}
 \left(
  1+\zeta\rho^{-r-1}
  +O_\infty(\rho^{-2r-2})
 \right),
\]
where
\begin{equation}\label{eq:nr-general-zeta}
 \zeta
 =
 \frac{\Theta}{a_R}
 \sum_iS_{r-2}(k|i)(r-k_i^2)
 \qquad (r\ge2),
\end{equation}
and \(\zeta=0\) when \(r=1\). Together with
\eqref{eq:nr-general-trace}, this shows that all radial computations
in the proof of Lemma~\ref{lem:nr-normal-form} remain valid, with the
present values of \(a_R\) and \(\zeta\).

It remains to analyze the angular part of the linearization. In an
\(h\)-orthonormal frame diagonalizing \(K\), the angular principal
tensor of \(D\mathcal E(F)\), after removing the common radial factor
\(\rho^{-2}\), is diagonal with eigenvalues
\begin{equation}\label{eq:nr-angular-eigenvalues}
 q_i(\rho)
 =
 \frac{S_{r-1}(\mu,\lambda|i)}
 {(1-k_iF/\rho)^2}.
\end{equation}
Denote this symmetric tensor by \(\mathsf Q_\rho\).

For each fixed \(\rho\), the value \(q_i(\rho)\) depends on the
principal direction only through the corresponding eigenvalue
\(k_i\). Hence \(\mathsf Q_\rho\) is a spectral function of \(K\). If
\(\nu_1,\ldots,\nu_d\) are the distinct eigenvalues of \(K\),
polynomial interpolation gives a polynomial \(p_\rho\), of degree at
most \(d-1\), such that
\[
 \mathsf Q_\rho=p_\rho(K).
\]
The interpolation coefficients depend on \(\rho\) and on the constant
numbers \(\nu_1,\ldots,\nu_d\), but not on the point of
\(\Gamma\). By the Cayley--Hamilton and Newton identities, \(p_\rho(K)\)
can be written as
\[
 \mathsf Q_\rho=\sum_{j=0}^{n-2}c_j(\rho)T_j(K),
\]
where every coefficient \(c_j(\rho)\) is constant on the link.

Each Newton tensor \(T_j(K)\) is divergence-free with respect to \(h\).
Since the eigenvalues of \(K\), and hence all \(S_j(K)\), are constant,
we conclude that
\begin{equation}\label{eq:nr-Q-divfree}
 \operatorname{div}_h\mathsf Q_\rho=0.
\end{equation}
Moreover, the eigenvalues \(q_i(\rho)\) are constant on \(\Gamma\).
Thus
\begin{equation}\label{eq:nr-Q-determinant}
 \det_h\mathsf Q_\rho
 =
 \prod_{i=1}^{n-1}q_i(\rho)
\end{equation}
is constant on the link for every fixed \(\rho\). Moreover,
strict admissibility and convergence to the conical coefficients imply
that \(\mathsf Q_\rho\) is positive definite for all sufficiently
large \(\rho\).

Define the radius-dependent angular metric \(h_\rho\) by
\[
 h_\rho^{-1}=\mathsf Q_\rho,
 \qquad\text{that is,}\qquad
 h_\rho^{ab}=\mathsf Q_\rho^{ab}.
\]
By \eqref{eq:nr-Q-determinant}, the density ratio
\[
 \frac{d\mu_{h_\rho}}{d\mu_h}
 =
 \bigl({\rm det}_{h}\mathsf Q_\rho\bigr)^{-1/2}
\]
is constant on the link. Combining this fact with
\(\operatorname{div}_h\mathsf Q_\rho=0\), we obtain, for every smooth
function \(u\) on the link,
\begin{align}
 \Delta_{h_\rho}u
 &=
 \frac{1}{\sqrt{\det h_\rho}}
 \partial_a\left(
  \sqrt{\det h_\rho}\,
  h_\rho^{ab}\partial_bu
 \right) \notag\\
 &=
 \frac{1}{\sqrt{\det h}}
 \partial_a\left(
  \sqrt{\det h}\,
  \mathsf Q_\rho^{ab}\partial_bu
 \right) \notag\\
 &=
 \mathsf Q_\rho^{ab}\nabla_a^h\nabla_b^hu.
 \label{eq:nr-Q-laplacian}
\end{align}
Thus the angular principal part of the linearization at \(F\) is
exactly a Laplace--Beltrami operator and introduces no additional
angular first-order term.

After the radial change of variables and division by the radial
principal coefficient, we denote the limiting normalized angular tensor
by $\mathsf Q_\infty$; explicitly,
\[
 \mathsf Q_\infty^{ab}
 =
 \frac{1}{a_R\beta^2}T^{ab},
 \qquad
 \beta=\frac{r+1}{2}.
\]
We therefore define the limiting link metric \(h_*\) by
\begin{equation}\label{eq:nr-h-star-general}
 h_*^{ab}
 =
 \frac{1}{a_R\beta^2}T^{ab}.
\end{equation}
Since \(T>0\), this defines a smooth Riemannian metric on the link.
Moreover, the eigenvalues of \(T\) are constant, so
\(\det_hT\) is constant. As above,
\[
 \Delta_{h_*}u
 =
 \frac{1}{a_R\beta^2}
 T^{ab}\nabla_a^h\nabla_b^hu.
\]

We next verify that there are no further angular first-order terms in
the linearization at the radial graph \(F\). Since \(F\) depends only
on \(\rho\), every term containing an angular derivative of \(F\)
vanishes. The only remaining possible contributions involving
\(\nabla K\) are contractions of the form
\(\operatorname{tr}(P(K)\nabla_YK)\) and
\(\sum_iP(k_i)\ip{(\nabla_{e_i}K)e_i}{Y}_h\), with polynomial weights
representing the spectral functions of the constant eigenvalues; both
vanish by Lemma~\ref{lem:spectral-cancellation}. The radial--angular principal
coefficients vanish at $F$ because the graph metric, its shape operator,
and its Newton tensor are block diagonal in the radial--angular
splitting. The algebraic angular gradient terms contain no angular
derivative of $F$. Thus these are tensorial cancellations; no assumption
that $\nabla K=0$ is being made. Hence
the angular part of the linearization at \(F\) is precisely the
Laplacian associated with the radius-dependent effective metric.

After setting
\[
 s=\rho^\beta,
 \qquad
 A_0=a_R\beta^2,
 \qquad
 \eta=\frac{\gamma\beta}{A_0},
\]
the effective metric determined by the full principal symbol has the
expansion
\begin{equation}\label{eq:nr-general-metric-expansion}
 g
 =
 ds^2+s^2h_*+h_1+O_\infty(s^{-2}),
\end{equation}
up to mixed terms of relative order \(O_\infty(s^{-4})\), where
$h_1$ is independent of $s$. The angular connection cancellation
follows from \eqref{eq:nr-Q-laplacian}; one does not discard coordinate
first-order terms. Indeed, the \(h_1\)-term is produced
by the first correction of the radial approximation \(F\), while
\[
 \frac{f_t-F}{\rho}
 =
 O_\infty(\rho^{-2r-2})
 =
 O_\infty(s^{-4})
\]
shows that replacing \(F\) by \(f_t\) contributes only relative
\(O_\infty(s^{-4})\)-errors.

It follows that
\[
 |\nabla_gs|_g=1,
 \qquad
 |\nabla_g^2s^2-2g|_g=O(s^{-2}),
 \qquad
 \Delta_gs=\frac{n-1}{s}+O(s^{-3}).
\]
Exactly as in Lemma~\ref{lem:nr-normal-form}, the normalized equation
for \(w=f_1-f_2\) becomes
\[
 \Delta_gw-\eta s\partial_sw+\frac{\ell}{s}\partial_sw
 +\frac{\gamma}{A_0}w
 =
 E_1\cdot\nabla_gw+E_0w,
\]
where
\[
 |E_1|_g=O(s^{-3}),
 \qquad
 |E_0|=O(s^{-2}),
\]
and
\[
 \ell
 =
 \eta\zeta+
 \frac{\beta-1+n-r}{\beta}-(n-1)
 =
 \eta\zeta-\frac{n(r-1)}{r+1}.
\]

Finally, put
\[
 \mathcal R=\sqrt{2\eta}\,s,
 \qquad
 \widehat g=2\eta g,
 \qquad
 v=\mathcal R^{\ell/2}w.
\]
The conjugation by \(\mathcal R^{\ell/2}\) cancels the \(\ell/\mathcal R\)-drift and gives
\[
 \left|
 \left(
  \Delta_{\widehat g}
  -\frac{\mathcal R}{2}
   \langle\nabla_{\widehat g}\mathcal R,
          \nabla_{\widehat g}\,\cdot\,\rangle
  +\varpi_*
 \right)v
 \right|
 \le
 C\mathcal R^{-2}\left(
  |v|+|\nabla_{\widehat g}v|
 \right),
\]
with
\[
 \varpi_*=\frac1{r+1}+\frac \ell4.
\]
Furthermore,
\[
 |\nabla_{\widehat g}\mathcal R|_{\widehat g}=1,
 \qquad
 |\nabla_{\widehat g}^2\mathcal R^2-2\widehat g|_{\widehat g}
 \le C\mathcal R^{-2}.
\]
Thus \((E,\widehat g,\mathcal R)\) is a weakly conical end. By
Lemma~\ref{lem:nr-decay}, \(v\) and all its derivatives decay faster
than every power of \(\mathcal R\).

The proof of Theorem~\ref{thm:nr-circular} now applies verbatim: the
asymptotic trace of \(v\) given by Theorem~\ref{thm:bernstein} vanishes
because of the superalgebraic decay, and therefore \(v=0\) on a
sufficiently distant end. Hence
\(w=0\) there, and the two hypersurface ends coincide outside a
compact set.
\end{proof}

\begin{remark}
The proposition includes noncircular cones. For
$1\leq r\leq n-2$, consider
\[
 \Gamma=\Sph^{n-2}(a)\times\Sph^1(b)\subset\Sph^n,
 \qquad
 a^2+b^2=1,\qquad k=b/a.
\]
With a suitable orientation its principal curvatures are $k$,
with multiplicity $p=n-2$, and $-1/k$, with multiplicity one.
Hence
\[
 S_j(0,k,\ldots,k,-1/k)
 =
 \binom pj k^j-\binom p{j-1}k^{j-2}>0
 \qquad(1\leq j\leq r)
\]
whenever
\(
 k^2>\frac{r}{n-1-r}.
\)
Proposition~\ref{prop:nr-constant-curvatures} then gives uniqueness
of $S_r$-shrinking ends asymptotic to the cone over $\Gamma$, provided
such ends exist. For $r=1$ this is contained in Wang's theorem; for
$r\ge2$ the existence of such ends is an open problem.
\end{remark}

\begin{remark}\label{rem:nr-general-cones}
For an arbitrary regular cone, Lemma~\ref{lem:nr-decay} still
gives superalgebraic decay of the difference. The continuation argument above,
however, does not apply directly. Even if the cone is strictly
$r$-admissible, normalization of the radial principal coefficient
produces the drift
\[
 -\frac{\gamma}
 {\beta S_{r-1}(K(\theta))}
 \,s\partial_s,
\]
whose coefficient may depend on the angular variable.
This is a principal term, not an error allowed in
\eqref{eq:nr-bernstein-inequality}.

Further angular dependence may occur in the coefficient $\ell$ of
the $s^{-1}$ drift. Conjugation by a power with an angularly varying exponent
introduces logarithmic derivatives and does not supply the required
estimate. Finally, if the limiting cone lies on the boundary of
$\Ga_r$, positivity of the scaled principal symbol must itself
be addressed. Thus uniqueness for general regular cones remains
outside the scope of this argument.
\end{remark}

\section{Geometric consequences}\label{sec:geometry}

\begin{proposition}\label{prop:geometry}
For $R>0$, set $u_{\sigma,R}(s)=R^{-1}u_\sigma(Rs)$.  For every
$0<a<A<\infty$ and integer $k\ge0$,
\[
\|u_{\sigma,R}-\sigma(\cdot)\|_{C^k([a,A])}
=O(R^{-(r+1)}).
\]
Furthermore,
\[
\lim_{x\to\infty}x\lambda(x)
=\frac1{\sigma\sqrt{1+\sigma^2}},
\qquad
|A_{\Sigma_\sigma}(x)|=O(x^{-1}).
\]
\end{proposition}

\begin{proof}
For each $j\ge0$, differentiation of the rescaled profile contributes
$R^{j-1}$ to the $j$th derivative. The differentiated expansion in
Theorem~\ref{thm:expansion-intro}, uniformly for $s\in[a,A]$, therefore
gives the asserted $C^k$ estimate. In particular, the end is smoothly
asymptotic to its cone with multiplicity one. Since
\[
\lambda=\frac1{u_\sigma\sqrt{1+(u_\sigma')^2}},
\]
the asserted limit follows from $u_\sigma/x\to\sigma$ and
$u_\sigma'\to\sigma$.  Finally, set
\[
\Phi_\sigma=\frac{n-r}{ru_\sigma}
 +\kappa m_\sigma^{r-1}\Psi_\sigma.
\]
Theorem~\ref{thm:main} gives $\Phi_\sigma=O(x^{-1})$.  Since
\(
\mu=-\frac{\Phi_\sigma}{\sqrt{1+(u_\sigma')^2}},
\)
we obtain $\mu=O(x^{-1})$ and hence $|A|=O(x^{-1})$.
\end{proof}

The expansion also determines the first correction to the transversal
principal curvature.

\begin{proposition}\label{prop:transversal-sharp}
As $x\to\infty$,
\begin{equation}\label{eq:lambda-sharp}
\lambda(x)=\frac{1}{\sigma\sqrt{1+\sigma^2}}x^{-1}
-\frac{c_1\bigl(1+(1-r)\sigma^2\bigr)}
{\sigma^2(1+\sigma^2)^{3/2}}x^{-r-2}
+O(x^{-2r-3}).
\end{equation}
Consequently,
\[
|A_{\Sigma_\sigma}(x)|
=\frac{\sqrt{n-1}}{\sigma\sqrt{1+\sigma^2}}x^{-1}
+O(x^{-r-2}).
\]
\end{proposition}

\begin{proof}
Theorem~\ref{thm:expansion-intro}, with derivatives through order two, gives
\[
u=\sigma x+c_1x^{-r}+O(x^{-2r-1}),\qquad
u'=\sigma-r c_1x^{-r-1}+O(x^{-2r-2}).
\]
Therefore
\[
\sqrt{1+(u')^2}=\sqrt{1+\sigma^2}
\left(1-\frac{r\sigma c_1}{1+\sigma^2}x^{-r-1}
+O(x^{-2r-2})\right),
\]
and expansion of $\lambda=(u\sqrt{1+(u')^2})^{-1}$ proves
\eqref{eq:lambda-sharp}.  Finally,
$|A|^2=(n-1)\lambda^2+\mu^2$ and
$\mu=O(x^{-r-2})$ by \eqref{eq:meridian-expansion}; taking the positive
square root gives the last assertion.
\end{proof}

\begin{proposition}
Let $\omega_{n-1}=|\Sph^{n-1}|$.  Then
\[
\operatorname{Vol}_n(\Sigma_\sigma\cap\{b\le x\le R\})
=\frac{\omega_{n-1}}n\sigma^{n-1}\sqrt{1+\sigma^2}\,R^n+o(R^n).
\]
Moreover, for every $b>0$,
\[
\int_{\Sigma_\sigma\cap\{x\ge b\}}|A|^n\dd\mu=\infty.
\]
More precisely,
\[
\int_{\Sigma_\sigma\cap\{b\le x\le R\}}|A|^n\dd\mu
=\frac{\omega_{n-1}(n-1)^{n/2}}
{\sigma(1+\sigma^2)^{(n-1)/2}}\log R+O(1).
\]
\end{proposition}

\begin{proof}
The area form is
\[
\dd\mu=u_\sigma^{n-1}\sqrt{1+(u_\sigma')^2}\dd x\dd\omega.
\]
The volume asymptotics follow from $u_\sigma/x\to\sigma$ and
$u_\sigma'\to\sigma$. By Proposition~\ref{prop:transversal-sharp}
and Theorem~\ref{thm:expansion-intro}, the curvature density satisfies
\[
|A|^n u_\sigma^{n-1}\sqrt{1+(u_\sigma')^2}
=\frac{(n-1)^{n/2}}{\sigma(1+\sigma^2)^{(n-1)/2}}x^{-1}
+O(x^{-r-2}).
\]
The remainder is integrable at infinity. Integration in $x$ and over
$\Sph^{n-1}$ proves the logarithmic asymptotic formula and the divergence.
\end{proof}

We finally record what can be said about the profile away from
infinity. Since the right-hand side of \eqref{eq:profile} is real
analytic on $\{u>0\}$ for every $x\in\R$, the point $x=0$ plays no
special role for the equation.

\begin{proposition}\label{prop:maximal-graph}
The profile $u_\sigma$ extends uniquely to a maximal positive solution of
\eqref{eq:profile} on an interval $(x_\sigma,\infty)$ with
$x_\sigma\in[-\infty,0]$. If $x_\sigma>-\infty$, then
$(u_\sigma(x),u_\sigma'(x))$ leaves every compact subset of
$(0,\infty)\times\R$ as $x\downarrow x_\sigma$. Moreover:
\begin{enumerate}
\item if $x_\sigma<0$, then $u_\sigma(0)>0$, $\Psi_\sigma<0$ on
$[0,\infty)$, and $\Sigma_\sigma\cap\{x\ge0\}$ is $r$-admissible;
\item if $x_\sigma=0$, then the limit
$L_\sigma=\lim_{x\downarrow0}u_\sigma(x)/x$ exists in $(\sigma,\infty]$;
if $L_\sigma<\infty$, then $u_\sigma(x)\to0$ and
$u_\sigma(\rho s)/\rho\to L_\sigma s$ locally uniformly for $s>0$ as
$\rho\downarrow0$.
\end{enumerate}
\end{proposition}

\begin{proof}
The existence of a maximal solution and the behavior at a finite endpoint
are the standard continuation theorem for the first-order system
$u'=p$, $p'=F(x,u,p)$, whose right-hand side is smooth on
$\R\times(0,\infty)\times\R$. Since $u_\sigma$ is defined on
$(0,\infty)$, we have $x_\sigma\le0$.

(i) If $x_\sigma<0$, then $u_\sigma(0)>0$ and
$\Psi_\sigma(0)=-u_\sigma(0)<0$, while $\Psi_\sigma<0$ on $(0,\infty)$
by \eqref{eq:PsiJ}. Since $\lambda>0$ on a graph,
Proposition~\ref{prop:admissibility} gives admissibility.

(ii) By \eqref{eq:PsiJ}, $(u_\sigma/x)'=\Psi_\sigma/x^2<0$ on
$(0,\infty)$, so $u_\sigma/x$ has a limit in $(\sigma,\infty]$ as
$x\downarrow0$. If $L_\sigma<\infty$, then
$u_\sigma(x)=x(u_\sigma(x)/x)\to0$, and for every compact interval
$[a,A]\Subset(0,\infty)$,
\[
\sup_{s\in[a,A]}
\left|\frac{u_\sigma(\rho s)}\rho-L_\sigma s\right|
\le A\sup_{y\in[\rho a,\rho A]}
\left|\frac{u_\sigma(y)}y-L_\sigma\right|\longrightarrow0.
\qedhere
\]
\end{proof}

\begin{proposition}[Obstruction at a vertical tangent]
\label{prop:vertical-obstruction}
Let $2\le r\le n-1$ and $c\ne0$. On a $C^2$ regular rotational
immersion satisfying $S_r=c\langle X,N\rangle$, a vertical meridian
tangent at positive radius can occur only in the plane $x=0$.
More precisely, suppose that a graphical solution approaches a finite
point $(x_*,y_*)$ with $x_*\ne0$, $y_*>0$, and $|u'|\to\infty$.
Then it has no $C^2$ regular continuation as a solution of the same
equation through that point. With the graphical orientation,
\begin{equation}\label{eq:vertical-curvature-blowup}
 \lambda\to0,\qquad
 \mu\lambda^{r-1}\longrightarrow
 \frac{c\,x_*\epsilon}{\binom{n-1}{r-1}}\ne0,
 \qquad
 \epsilon=\lim\frac{u'}{\sqrt{1+(u')^2}}\in\{-1,1\}.
\end{equation}
In particular, the meridian curvature diverges in absolute value.
\end{proposition}
\begin{proof}
Parametrize a regular meridian by arclength $\tau$, so that
$\dot x^2+\dot y^2=1$, and choose $N=(\dot y,-\dot x\omega)$.
At a vertical tangent, $\dot x=0$ and $\dot y=\pm1$. Hence
$\lambda=\dot x/y=0$ and $\langle X,N\rangle=\pm x$.
The meridian curvature $\mu$ is finite on a $C^2$ regular immersion.
Since
\[
 S_r=\binom{n-1}{r-1}\mu\lambda^{r-1}
     +\binom{n-1}{r}\lambda^r,
\]
we have $S_r=0$ for $r\ge2$, and the soliton equation forces $x=0$.

For the graphical limit in the statement,
$\lambda=(u\sqrt{1+(u')^2})^{-1}\to0$.
Continuity of $u'$ and $|u'|\to\infty$ imply that its sign is eventually
constant, so $u'/\sqrt{1+(u')^2}\to\epsilon$ and
\[
 \langle X,N\rangle
 =\frac{xu'-u}{\sqrt{1+(u')^2}}\longrightarrow x_*\epsilon.
\]
The formula for $S_r$ gives \eqref{eq:vertical-curvature-blowup}.
Thus $|\mu|$ grows like a nonzero constant times
$\lambda^{1-r}$, excluding a $C^2$ regular continuation.
\end{proof}

\begin{remark}\label{rem:numerical-maximal}
(Numerical.) In the computed cases the profiles cross $x=0$ at positive
height and become steep near a negative left endpoint $x_\sigma$, at a
positive radius; for $r\ge2$, Proposition~\ref{prop:vertical-obstruction}
then excludes a $C^2$ regular continuation. The angular defect can become
positive before $x_\sigma$ is reached, so $r$-admissibility may be lost
within the graphical region. Approximate values:
\begin{center}
\begin{tabular}{cccccc}
$n$ & $r$ & $\sigma$ & $u_\sigma(0)$ & $x_\sigma$ & $\{x>x_\sigma:\Psi_\sigma>0\}$\\
\hline
3 & 1 & 1 & 1.187 & $-1.68$ & $(-1.68,-1.43)$\\
3 & 2 & 1 & 0.456 & $-1.37$ & $(-1.37,-0.93)$\\
4 & 2 & 0.5 & 0.937 & $-2.48$ & $\emptyset$\\
5 & 3 & 2 & 0.333 & $-0.81$ & $(-0.81,-0.49)$
\end{tabular}
\end{center}
Here $\kappa=1$, and the computations are as in
Figure~\ref{fig:rotational-profiles}.
\end{remark}

For each $b\in\R$ with $b>x_\sigma$, the tail
$\Sigma_\sigma\cap\{x\ge b\}$ is proper and complete in the direction
$x\to\infty$: properness follows from $|X|\ge x$, and every curve escaping
through $x\to\infty$ has infinite length. Completeness of the whole hypersurface is not established by these
estimates. A possible continuation must solve the same geometric
equation; changing parametrization does not remove the obstruction of
Proposition~\ref{prop:vertical-obstruction}. For $r=1$, after
the normalization $\kappa=1/2$ of Kleene--M\o ller, the profile is defined
on $[0,\infty)$, it is strictly convex there with $0<u_\sigma'<\sigma$, and
$u_\sigma(0)<\sqrt{2(n-1)}$ \cite[Theorem~3(i),(iv), p.~3947]{KleeneMoller2014};
in particular $x_\sigma<0$. Moreover, the maximal curve containing the
graph of $u_\sigma$ is not embedded, so these ends are not parts of
complete embedded self-shrinkers
\cite[Theorem~3(v) and Corollary~5, pp.~3947 and 3956]{KleeneMoller2014}.

\section{Conclusions and open problems}\label{sec:finer}

For every $1\le r\le n-1$ and every positive conical slope we have
constructed a smooth embedded rotational $S_r$-shrinking end. The strict
negativity of the angular defect gives $r$-admissibility, the profile is
unique under the sole assumption $u(x)/x\to\sigma>0$, and its complete
differentiable asymptotic expansion determines the geometry at infinity.
Proposition~\ref{prop:gauss-endpoint} shows that for $r=n$ the only
rotational graphical solution with a positive conical asymptote is the
cone itself. On compact slope intervals the profiles depend real
analytically on the slope, in weighted H\"older spaces controlling the
full conical remainder, and they foliate distant exterior regions.
Theorem~\ref{thm:nr-circular} removes the symmetry assumption from the
uniqueness statement under smooth multiplicity-one convergence to a
circular cone, and Proposition~\ref{prop:nr-constant-curvatures} extends
it to strictly admissible cones whose links have constant principal
curvatures. These results identify the ends outside a compact set, but
they do not classify their interior continuations. We conclude with
several questions.

\subsection*{Intersections of the rotational family}
Proposition~\ref{prop:slope-dependence} and
Corollary~\ref{cor:analytic-exterior-foliation} show that the profiles are
ordered on distant exterior regions. Numerically, this ordering does not
extend to all $x>0$.

\begin{remark}\label{rem:numerical-crossing}
(Numerical.) For
$\kappa=1$, $(n,r)\in\{(3,1),(3,2),(4,1),(5,2),(5,3),(6,4)\}$, and
consecutive slopes in $\{0.25,0.5,1,2,4\}$, the profiles $u_\sigma$ and
$u_{\sigma'}$ with $\sigma<\sigma'$ have exactly one detected
intersection in $\{x>0\}$. For instance, for $n=3$ the intersection points are approximately
$x=0.47$ ($r=1$) and $x=0.47$ ($r=2$) when $(\sigma,\sigma')=(0.5,1)$, and
$x=0.19$ ($r=1$) and $x=0.17$ ($r=2$) when $(\sigma,\sigma')=(1,2)$. In
all these cases $x_\sigma<0$, and the computed values of
$u_\sigma(0)$ decrease along the sampled slopes.
\end{remark}

If one proves that $x_\sigma<0$ and that $\sigma\mapsto u_\sigma(0)$ is
strictly decreasing, then Proposition~\ref{prop:slope-dependence} and the
intermediate value theorem show that $u_\sigma$ and $u_{\sigma'}$
intersect in $(0,R(\sigma,\sigma'))$. In terms of the variation
$J_\sigma=\partial_\sigma u_\sigma$, which satisfies the linearized
equation
\[
 J_\sigma''=F_u(x,u_\sigma,u_\sigma')J_\sigma
            +F_p(x,u_\sigma,u_\sigma')J_\sigma'
\]
and $J_\sigma=x+O(x^{-r})$ at infinity
(Corollary~\ref{cor:analytic-exterior-foliation}), global ordering of the
family is equivalent to $J_\sigma>0$ on $(0,\infty)$ for every $\sigma$.
Indeed, ordering gives $J_\sigma\ge0$; at an interior zero $x_0$ of a
nonnegative $J_\sigma$ we would have $J_\sigma'(x_0)=0$, and uniqueness
for the linear equation would force $J_\sigma\equiv0$, which is
impossible. Conversely, integrating $J_\sigma>0$ in $\sigma$ gives
ordering. More precisely, an intersection forces a zero of the
variation: if $u_\sigma(x_0)=u_{\sigma'}(x_0)$ for some $x_0>0$ and
$\sigma<\sigma'$, then
\[
 0=\int_\sigma^{\sigma'}J_\tau(x_0)\,d\tau ,
\]
so $J_\tau(x_0)\le0$ for some $\tau\in(\sigma,\sigma')$; since
$J_\tau(x)=x+O(x^{-r})>0$ for large $x$, the function $J_\tau$ has a zero
in $(0,\infty)$. Thus each intersection observed in
Remark~\ref{rem:numerical-crossing} yields a zero of $J_\tau$ for some
intermediate slope, and $J_\tau>0$ fails for that $\tau$.

\begin{question}\label{ques:global-ordering}
Do the profiles $u_\sigma$ and $u_{\sigma'}$, $\sigma\ne\sigma'$, always
intersect in $\{x>0\}$, and, if an intersection exists, is it unique?
\end{question}

\begin{question}\label{ques:variation-zero}
For each $\sigma>0$, does the variation
$J_\sigma=\partial_\sigma u_\sigma$ have exactly one zero in
$(0,\infty)$?
\end{question}

The two questions concern finite differences and infinitesimal
variations, respectively. We do not assert their equivalence.
Such an equivalence would require a separate comparison argument for
the profile equation, including control of zeros approaching the
boundary of the graphical interval.

\subsection*{The maximal graphical interval}
The construction concerns the conical end, and
Proposition~\ref{prop:maximal-graph} only describes the maximal graphical
interval qualitatively.

\begin{question}\label{ques:finite-end}
Let $2\le r\le n-1$. Is $x_\sigma<0$ for every $\sigma>0$?
When $x_\sigma$ is finite, does $u_\sigma$ approach a positive finite
radius with $|u_\sigma'|\to\infty$, and what is the precise rate of
curvature blow-up? At which points, if any, does the generating curve
leave the closure of the $r$-admissible region? What other endpoint
behaviors are possible, including approach to the axis or a second end?
If a weak or singular continuation is considered, which notion of
solution is appropriate and under what conditions does it exist?
\end{question}

If the endpoint has nonzero $x_\sigma$ and positive finite radius with
a vertical limiting tangent, a smooth continuation is ruled out by
Proposition~\ref{prop:vertical-obstruction}; reparametrization by
arclength cannot regularize the meridian curvature. Endpoints with
$x_\sigma=0$ or vanishing radius require separate analysis and are not
classified by that proposition.

For $r=1$, the classical results of Kleene--M\o ller give
$x_\sigma<0$ and a nonembedded maximal generating curve
\cite[Theorem~3, p.~3947, and Corollary~5, p.~3956]{KleeneMoller2014};
their proof uses geodesics of the Angenent metric. The obstruction at a
positive-radius vertical tangent is different for $r\ge2$, since the
coefficient of the meridian curvature vanishes when $\lambda=0$.

\subsection*{General cones}

\begin{question}\label{ques:general-cones}
For $2\le r\le n-1$, does uniqueness at infinity hold for smoothly
asymptotically conical $S_r$-shrinking ends over every strictly
$r$-admissible regular cone, without assuming that the principal
curvatures of its link are constant?
\end{question}

Lemma~\ref{lem:nr-decay} shows that the difference of two such ends
vanishes to every algebraic order for an arbitrary regular cone. The
remaining obstruction is therefore not asymptotic flatness but unique
continuation at infinity for the effective operator. As explained in
Remark~\ref{rem:nr-general-cones}, normalization of its radial principal
coefficient generally produces an angularly dependent leading drift,
which is not covered by the weakly conical almost-eigenfunction theory of
Theorem~\ref{thm:bernstein}. If the limiting cone lies on the
boundary of the G{\aa}rding cone, one must in addition address the
degeneration of the principal symbol.

\section*{Funding}
The authors were partially supported by the Brazilian National Council
for Scientific and Technological Development (CNPq), through Grant
303118/2022-9 to H.A., Grant 304381/2026-8 to M.B., and Grant
306189/2023-2 to G.S.N. All authors were also partially supported by the
Coordination for the Improvement of Higher Education Personnel (CAPES),
Brazil, Finance Code 001.

\section*{Data availability statement}
No data are associated with this manuscript. Figure~\ref{fig:rotational-profiles}
and Remarks~\ref{rem:numerical-maximal} and~\ref{rem:numerical-crossing}
are based on numerical solutions of the terminal problems described in the
caption of Figure~\ref{fig:rotational-profiles}; they are not certified
and no proof relies on them.

\section*{Conflict of interest statement}
The authors declare that they have no conflict of interest.

\section*{AI disclosure}
A generative AI tool was used to assist with language editing, organization of the exposition, and discussion and checking of selected algebraic calculations. The authors take full responsibility for the mathematical statements, proofs, and references.

\end{document}